\documentclass[reqno]{amsart}
\usepackage{graphicx} 
\usepackage{tikz}
\usepackage{mathrsfs}
\usetikzlibrary{arrows.meta,positioning}
\usepackage{amsmath, amsthm, amsfonts, amssymb, bm,bbm}
\usepackage{comment}
\usetikzlibrary{decorations.pathreplacing}
\usepackage{booktabs}
\usepackage[english]{babel}
\usepackage[title]{appendix}
\usepackage{setspace}
\usepackage{graphicx}
\usepackage{enumitem}
\usepackage[labelfont=bf, font={small,it}]{caption}
\usepackage{float}
\usepackage{mathtools}
\usepackage{array}
\usepackage{xcolor}
\usepackage{hyperref}
\usepackage{float}
\usepackage{subcaption}
\usepackage{fancyhdr}
\usepackage{mdframed}
\usepackage{listings}

\usepackage{amsthm}
\makeatletter
\def\th@plain{%
  \thm@notefont{}
  \itshape 
}
\def\th@definition{%
  \thm@notefont{}
  \normalfont 
}
\makeatother

\usepackage{stmaryrd}
\usepackage{setspace}
\numberwithin{equation}{section}

\newtheorem{thm}{Theorem}[section]
\newtheorem{cor}[thm]{Corollary}
\newtheorem{define}[thm]{Definition}
\newtheorem{lemma}[thm]{Lemma}
\newtheorem{prop}[thm]{Proposition}
\newtheorem{rmk}[thm]{Remark}

\DeclareMathOperator{\R}{\mathbb{R}}
\DeclareMathOperator{\N}{\mathbb{N}}

\DeclareMathOperator{\Q}{\mathbb{Q}}
\DeclareMathOperator{\Z}{\mathbb{Z}}
\DeclareMathOperator{\E}{\mathbb{E}}
\DeclareMathOperator{\PP}{\mathbb{P}}
\DeclareMathOperator{\PPi}{\mathscr{P}}

\DeclareMathOperator{\A}{\mathcal{A}}
\DeclareMathOperator{\AAA}{\mathfrak{A}}

\DeclareMathOperator{\D}{\mathbb{D}}
\DeclareMathOperator{\ZZ}{\tilde{Z}}
\DeclareMathOperator{\SSS}{\mathcal{S}}

\DeclareMathOperator{\W}{\mathcal{W}}

\DeclareMathOperator{\indic}{\mathbbm{1}}
\DeclareMathOperator{\G}{\mathcal{G}}
\DeclareMathOperator{\F}{\mathcal{F}}
\DeclareMathOperator{\h}{\mathcal{H}}

\newcommand{\dq}{\overset{d}{=}}
\newcommand{\bd}{\overset{\beta}{\searrow}}
\newcommand{\lb}{\llbracket}
\newcommand{\rb}{\rrbracket}

\title{Convergence from the Log-Gamma Polymer to the Directed Landscape}

\author{Xinyi Zhang}

\begin{document}

\begin{abstract}

We define the log-gamma sheet and the log-gamma landscape in terms of the $2$-parameter and $4$-parameter free energy of the log-gamma polymer model and prove that they converge to the Airy sheet and the directed landscape, which are central objects in Kardar-Parisi-Zhang (KPZ) universality class. Our proof of the convergence to the Airy sheet relies on the invariance of free energy through the geometric RSK correspondence and the monotonicity of the free energy. To upgrade the convergence to the directed landscape, tail bounds in both spatial and temporal directions are required. However, due to the lack of scaling invariance in the discrete log-gamma polymer---unlike the Brownian setting of the O’Connell–Yor model---existing on-diagonal fluctuation bounds are insufficient. We therefore develop new off-diagonal local fluctuation estimates for the log-gamma polymer. 

\end{abstract}

\maketitle

\section{Introduction}
The study of stochastic growth models has long been one of the central topics in probability theory \cite{barabasi1995fractal,comets2004probabilistic,kipnis2013scaling,halpin1995kinetic,meakin1998fractals,majumdar2007random,spohn2006exact}. The well-known central limit theorem characterizes the magnitude of the fluctuation of a sum of independent and identically distributed (i.i.d.) random variables with $1/2$ exponent. The limiting distribution of the fluctuation after proper scaling will be Gaussian. However, with certain dependence of variables introduced to the system, we observe totally different limiting behavior. As suggested by various experimental evidence and numerical simulations \cite{takeuchi2010universal,maunuksela1997kinetic,miettinen2005experimental,wang2003universality}, a large class of random growth models including random growth interfaces, interacting particle systems, and directed polymer models have fluctuations of $1/3$ scaling exponent and have completely different limiting distributions. These random growth models fall into  the \textit{Kardar–Parisi–Zhang (KPZ) universality class}. 

There have been many important breakthroughs in the last twenty-five years in understanding the KPZ universality, including the work \cite{dauvergne2022directed} where Dauvergne, Ortmann and Vir\'ag constructed central objects in the KPZ universality class: the Airy sheets and the directed landscape. Not only do the Airy sheets and the directed landscape capture the intricate geometry of random growth phenomena but they are also believed to be the final limiting objects of growth models in the KPZ universality class. The Airy sheets and the directed landscape have been identified as scaling limits for a few zero-temperature models. Notable examples include Poissonian, exponential, and geometric Last Passage Percolation, as demonstrated in the works of Dauvergne-Nica-Virág \cite{dauvergne2021scaling} and \cite{dauvergne2023uniform}. 

However, proving such convergences are more challenging in positive temperature models and hence there are few results. The challenges of the positive temperature models come from two reasons. First, the robust determinantal structure that is present in the solvable zero temperature models is lost at positive temperature, and replaced by exact formulas that have generally only allowed for one-point asymptotic results. Second, there is no direct metric composition law as in the zero-temperature models and in the limit of the directed landscape; although in some cases, such as the KPZ equation and the log-gamma polymer, this absence of direct metric composition law is mitigated by a variational formula of the free energy. Despite these challenges, there are still notable results. Namely, the solutions to the KPZ equation by Wu \cite{wu2023kpz} and the colored ASEP by Aggarwal, Corwin, and Hegde \cite{aggarwal2024scaling} and the recent results by Dauvergne and Zhang that upgrade the KPZ fixed point convergence to the Directed landscape convergence \cite{dauvergne2024characterization}.

In this paper, we prove that the log‐gamma polymer, another positive temperature model  introduced in \cite{Sepp_l_inen_2012}, converges to the directed landscape. Our analysis builds on the geometric Robinson–Schensted–Knuth (RSK) correspondence, which plays a crucial role in relating the free energy of the log‐gamma polymer to the corresponding free energy on the dual line ensemble. The continuous limit of the geometric RSK correspondence preserves certain polymer free energy and the zero-temperature limit of the geometric RSK correspondence, also known as the melon map in \cite{dauvergne2022directed}, preserves certain last passage values. Motivated by how these invariance principles are used in \cite{wu2023kpz} and \cite{dauvergne2022directed} to prove convergence to directed landscape, we exploit the same invariance property of the discrete geometric RSK correspondece to obtain the key equality (\ref{equality}). This equality together with the variational formula (\ref{equation_variational_f}) serves as the main instrument to prove the convergence to the Airy sheet as it enables us to disintegrate the free energy from the bottom-left point to the top-right point of the line ensemble into two parts that only pertain the information of few top curves of the line ensemble. Similar to \cite{wu2023kpz}, we also leverage the monotonicity properties of Busemann functions to show that the contribution coming from the bottom part of the line ensemble becomes negligible in the scaling limit.

Due to lack of a nice coupling between the KPZ line ensemble and the KPZ sheet at the time that \cite{wu2023kpz} was written, Wu worked with the O'Connell Yor polymer which is the pre-limit of the KPZ line ensemble and has nice Brownian environment. In our case, the situation simplifies as there already exists a nice coupling between the log-gamma polymer and log-gamma line ensemble through the geometric RSK correspondence and hence we only need to take limit once instead of twice. 

However, the continuous Brownian environment in the work of \cite{wu2023kpz} is not present in the log-gamma polymer model. Although the Brownian symmetry is approximated by the i.i.d. environment of the polymer, lack of Brownian scaling has created challenges in proving the tightness of the log-gamma landscape. It is no longer sufficient to prove a tail bound for local fluctuations of the free energy along the diagonal line. Untypical off-diagonal fluctuations are also needed. We solved this by generalizing the methods used in \cite{basu2024temporal} for the diagonal case to the off-diagonal case.

\subsection{Log-gamma polymer}
\begin{define}
An inverse-gamma random variable $X \sim \text{Ga}^{-1}(\theta)$ with parameter $\theta > 0$ is a continuous random variable with density given by
\begin{equation}
f_\theta(x) = \frac{\indic\{x > 0\}}{\Gamma(\theta)}x^{-\theta-1}\exp(-x^{-1}).
\end{equation}
\end{define}

Fix some $\theta > 0$. Let $\{d_{i,j}\}_{i,j \in \Z}$ be i.i.d inverse-gamma random variables with parameter $\theta$ and $(\Omega, \mathcal{F}, \PP)$ be the probability space on which $\{d_{i,j}\}_{i,j \in \Z}$ are defined. 

Let us use $\preceq$ to denote the order on $\Z^2$ where $(x_1,y_1) \preceq (x_2,y_2)$ if $x_1 < x_2$ or $x_1 = x_2, y_1 \leq y_2$. Let $u, v \in \Z^2$. An up-right lattice path $\pi$ connecting $u$ to $v$ is a set of vertices that can be ordered as $u =(x_1, y_1), \cdots, (x_n,y_n) = v \in \Z^2$ such that $(x_i,y_i) \preceq (x_{i+1},y_{i+1})$ and $x_{i+1} - x_i + y_{i+1}-y_i = 1$ for all $i= 1, \cdots, n-1$. Let $\Pi(u \rightarrow v)$ denote the set of up-right lattice path $\pi$ that connects $u$ to $v$. 

Let $\pi = (\pi_1, \cdots, \pi_k)$ where $\pi_i$ is an up-right path. We call $\pi$ a multipath if for all $1 \leq i \neq j \leq k$, $\pi_i \cap \pi_j = \emptyset$. Let $U$ and $V$ be two finite subsets of $\Z^2$ with the same cardinality. We call $(U,V)$ an endpoint pair if there exists a multipath $\pi = (\pi_1, \cdots, \pi_k)$ such that for every $u \in U$, there exists a path $\pi_i$ and $v \in V$ such that $\pi_i$ connects $u$ to $v$. Similarly, we use $\Pi[U \rightarrow V]$ to denote the set of multipaths $\pi$ that connects $U$ to $V$. 

\begin{define}[Polymer partition function and polymer free energy]\label{def_polymer_partition}
    Let $u, v \in \Z^2$. We define the partition function of the inverse-gamma polymer from $u$ to $v$ by
    \begin{equation}
        Z[u \rightarrow v] = \sum_{\pi \in \Pi[u \rightarrow v]} \prod_{(i,j) \in \pi} d_{i,j}. 
    \end{equation}
    Moreover, we call 
    \begin{equation}
        \log Z[u \rightarrow v]
    \end{equation} 
    the free energy of the inverse-gamma polymer from $u$ to $v$.

    For $u,v \in \Z^2$ such that $u \not \preceq v$, we assume $Z[u \rightarrow v] = 0$ and $\log Z[u \rightarrow v] = -\infty$.
\end{define}

\begin{define}
    Let $q, \sigma_p$ be some constants that only depend on the parameter $\theta$ and will be defined later in Section \ref{section3} Definition \ref{def_psi} and Thoerem \ref{thm_airy_line_ensemble} respectively. Let $\overline{x}_N = \lfloor N^{2/3}xq^{-2} \rfloor +1$, $\widehat{y}_N = \lfloor N^{2/3}yq^{-2} \rfloor + 2N$. We define the log-gamma sheet $h^N(x,y)$ to be the following scaled two-parameter polymer free energy:   
   \begin{equation}
    h^N(x,y) = 2^{-1/2}q\sigma_pN^{-1/3}\biggl[\log Z[(\overline{x}_N,1) \rightarrow (\widehat{y}_N,2N)]-p(\widehat{y}_N - \overline{x}_N + 2N)\biggr]
    \end{equation}
\end{define}

\begin{define}
    For $t,x \in \R$, we define the scaling operations $\overline{x}_N = \lfloor N^{2/3}xq^{-2} \rfloor + 1, t_N = \lfloor 2Nt \rfloor$. For $t,s,x,y \in \R$ with $t > s$, we define the log-gamma landscape $h^N(x,s;y,t)$ as the following scaled four-parameter polymer free energy:
    \begin{equation}
    \begin{split}
        h^N(x,s;y,t) &=2^{-1/2}q\sigma_pN^{-1/3} \biggl[ \log Z\bigl[(\overline{x}_N +s_N, s_N) \rightarrow  (\overline{y}_N +t_N -1,t_N-1)\bigr]  -p(\overline{y}_N - \overline{x}_N + 4N(t-s)) \biggr] 
    \end{split}
    \end{equation}
\end{define}

\subsection{Main Results}
Let us use $C(\R^2, \R)$ and $C(\R_+^4, \R)$ to denote the space of continuous functions on $\R^2$ and $\R_+^4: = \{(x,s;y,t) \in \R^4: s<t\}$, equipped with the topology of uniform convergence over compact subsets.

\begin{thm}[Airy sheet convergence]
    The continuous linear interpolation of $h^N(x,y)$ converges to the Airy sheet $S(x,y)$ in distribution as $C(\R^2,\R)$-random variables.
\end{thm}

\begin{thm}[Directed landscape convergence]
    The continuous linear interpolation of $h^N(x,s;y,t)$ converges to the directed landscape $\mathcal{L}(x,s;y,t)$ in distribution as $C(\R^4_+,\R)$-random variables.
\end{thm}

\subsection{Airy line ensemble, the Airy sheets, and the directed landscape}
In this section, we follow the presentation of \cite{dauvergne2022directed} and introduce the central limiting objects in the KPZ universality class: the Airy line ensemble, Airy sheets, and the directed landscape.

The finite-dimensional distributions of the stationary Airy line ensemble were first introduced by Pr\"ahofer and Spohn in \cite{prahofer2002scale}, where it was referred to as the ``multi-line Airy process". Subsequently, Corwin and Hammond \cite{corwin2014brownian} proved that this ensemble can be realized as a collection of continuous random functions indexed by $\N \times \R$, governed by the Brownian Gibbs property. 

\begin{define}
The stationary Airy line ensemble, denoted by $\tilde{\A} = \{\tilde{\A}_1 > \tilde{\A}_2 > \cdots\}$, is an infinite collection of random functions from $\R$ to $\R$ indexed by the natural numbers. The distribution of $\tilde{\A}$ is uniquely characterized by its determinantal structure: for any finite set $I = \{u_1, \dots, u_k\} \subset \mathbb{R}$, the point process on $I \times \mathbb{R}$ defined by $\{(s, \tilde{\A}_i(s)) : i \in \mathbb{N}, s \in I\}$ is a determinantal point process with kernel
\begin{equation}
K(s_1, x_1; s_2, x_2) =
\begin{cases}
\int_{0}^{\infty} e^{-z(s_1 - s_2)} \mathrm{Ai}(x_1 + z) \mathrm{Ai}(x_2 + z) \, dz & \text{if } s_1 \geq s_2, \\
-\int_{-\infty}^{0} e^{-z(s_1 - s_2)} \mathrm{Ai}(x_1 + z) \mathrm{Ai}(x_2 + z) \, dz & \text{if } s_1 < s_2,
\end{cases}
\end{equation}
where $\mathrm{Ai}$ denotes the Airy function.
\end{define}

As suggested by the name, the stationary Airy line ensemble $\tilde{\A}(t) = (\tilde{\A}_1(t),\tilde{\A}_2(t), \cdots )$ is stationary in $t$. Our focus, however, is on the parabolic Airy line ensemble. 

\begin{define}

The parabolic Airy line ensemble $\mathcal{A} = \{\A_1 > \A_2 > \cdots\}$ can be derived from $\tilde{\A}$ via
\begin{equation}
\A_i(x) := \tilde{\A}_i(x) - x^2.    
\end{equation}

Fix any real number $q > 0$, we define the Airy line ensemble of scale $q$, $\A^q = \{\A^q_1 > \A^q_2 > \cdots\}$, to be
\begin{equation}
    \A_i^q(x) = q^{-1}\A_i(q^2x).
\end{equation}
\end{define}

The Airy sheet was constructed by Dauvergne, Ortmann and Vir\'ag in \cite{dauvergne2022directed} via a last passage percolation framework on the Airy line ensemble. Thus, we need to first define the last passage time on a general line ensemble. 

\begin{define}\label{def_lpp}
    Let $f = (f_1, f_2, \cdots)$ be a sequence of functions from $\R$ to $\R$, we define the last passage time from $(x,\ell)$ to $(y,m)$ where $x \leq y$ and $\ell, m \in \Z$ such that $\ell \geq m$ as
    \begin{equation}
        f[(x,\ell) \xrightarrow{\infty} (y,m)] := \max_{x=t_{\ell+1} \leq t_{\ell} \leq t_{\ell-1} \leq \cdots \leq t_{m+1} \leq t_m = y} \sum_{j=m}^{\ell} \bigg( f_j(t_j) - f_j(t_{j+1}) \bigg).
    \end{equation}
    \end{define}

\begin{define}\label{def_fpp}
    Let $g = (g_1, g_2, \cdots)$ be a sequence of functions from $\R$ to $\R$, we define the backwards first passage time from $(x, 1)$ to $(y,m)$ where $x \leq y$ and $m \in \Z$ as
    \begin{equation}
        g[(x,1) \rightarrow_f (y,m)] := \min_{x=t_{0} \leq t_1 \leq \cdots \leq t_m = y} \sum_{j=1}^m \bigg(g_j(t_j) - g_j(t_{j-1})\bigg).
    \end{equation}
    \end{define}

We are now ready to define the Airy sheets and the directed landscape, both constructed in \cite{dauvergne2022directed}, through their characterizing properties. The fact that these properties uniquely determine the distributions of the Airy sheets and the directed landscape were proved in \cite{dauvergne2022directed}.

\begin{define}
The Airy sheet, denoted by $\SSS(x, y)$, is a random continuous function from $\mathbb{R}^2$ to $\R$ which can be uniquely characterized by the following properties:
\begin{enumerate}
    \item The distribution of $\SSS(\cdot + t, \cdot + t)$ is invariant under translations by $t$ for all $t \in \R$.
    \item There exists a coupling between $\SSS$ and an Airy line ensemble $\A$ where the marginal $\SSS(0, \cdot) = \A_1(\cdot)$, the Airy$_2$ process, and for all $x > 0$ and $y_1, y_2 \in \mathbb{R}$, the following limit holds almost surely:
    \begin{equation}
    \begin{split}
       & \lim_{k \to \infty} \A\left[ ( -2^{-1/2} k^{1/2} x^{-1/2}, k ) \xrightarrow{\infty} \left( y_2, 1 \right) \right] - \A\left[ ( -2^{-1/2} k^{1/2} x^{-1/2}, k ) \xrightarrow{\infty} \left( y_1, 1 \right)  \right] \\
        &= \SSS(x, y_2) - \SSS(x, y_1)
    \end{split}
    \end{equation}
\end{enumerate}
where $\A[(x,k) \xrightarrow{\infty} (y,1)]$ denotes the last passage time on the Airy line ensemble $\A$.

For any real number $q > 0$, we define the Airy sheet of scale $q$ to be
\begin{equation}
    \SSS^q(x,y) = q\SSS(q^{-2}x, q^{-2}y).
\end{equation}
\end{define}

\begin{define}
The directed landscape, denoted by $\mathcal{L}(x, s; y, t)$, is a continuous random function from $\mathbb{R}^4_+ = \{(x,s;y,t) \in \R^4: t > s\}$ to $\R$ that can be uniquely characterized by the following properties:
\begin{enumerate}
    \item For $t > s$, the marginal $\mathcal{L}(\cdot, s;\cdot,t)$ is distributed as an Airy sheet with scale $(t - s)^{1/3}$.
    \item For any finite set of disjoint intervals $\{(s_j, t_j)\}_{j=1}^m$, the functions $\{\mathcal{L}(\cdot,s_j; \cdot, t_j)\}_{j=1}^m$ are independent.
    \item For all $s < r < t$ and $x, y \in \mathbb{R}$, the following additivity property holds almost surely:
    \begin{equation}
    \mathcal{L}(s, x; t, y) = \max_{z \in \mathbb{R}} \left[ \mathcal{L}(s, x; r, z) + \mathcal{L}(r, z; t, y) \right].
    \end{equation}
\end{enumerate}
\end{define}

\subsection{Organization of the paper} Section 2 introduces invariance between the free energy of the polymer model the via the geometric RSK framework. Section 3 applies the results from Sections 2 and connects the free energy of the log-gamma polymer to the last passage percolation on Airy line ensemble. In Section 4  and 5, we present the convergence results by proving the tightness of the pre-limiting objects and showing that their distributional limits satisfy the characterizing properties of the Airy sheet and the directed landscape. In Section 6, we prove the key proposition that is used to establish the tightness of the log-gamma sheet.

\subsection{Notations}
Throughout this paper, we will use $\lb 1, n \rb$ to denote the set $\{1, 2, \cdots, n\}$ and $x \wedge y$ to denote the minimum of $x$ and $y$. We adopt the convention that an empty sum is interpreted as $0$ and an empty product as $1$.

\subsection{Acknowledgments}
The author sincerely thanks their advisor, Ivan Corwin, for suggesting this interesting problem, offering insightful guidance throughout the project, and providing valuable feedback on the manuscript. Thanks to Milind Hegde for many helpful discussions, and to Xiao Shen for insightful perspectives on the local fluctuations of the log-gamma free energy. The author is grateful to Zongrui Yang, Jiyue Zeng, and Alan Zhao for their valuable discussions as well. This research was partially supported by Ivan Corwin's National Science Foundation grant DMS:2246576 and Simons Investigator in Mathematics award MPS-SIM-00929852.

\section{ Geometric RSK correspondence: invariance of free energy}
In this section, we will consider the directed polymers with a deterministic environment and introduce some basic definitions and properties. The main ingredient in our analysis is the invariance of discrete free energy under the discrete geometric Pitman transform, also known as the geometric RSK correspondence.

Instead of $\Z^2$, let us restrict ourself to an infinite strip of the first quadrant, i.e., $\Z_{\geq 1} \times \lb 1, n \rb$ for some fixed $n.$ In order to be consistent with the line ensemble notations in the literature and notations in \cite{corwin2020invariance}, we always relabel the coordinate $(x,i)$ in $\Z_{\geq 1} \times \lb 1 , n \rb$ as $(x, n - i +1)$ in line ensemble environment, but we still consider the order $\preceq$ with respect to the original coordinate system, i.e. $(x_1,y_1) \preceq (x_2,y_2)$ if $x_1 < x_2$ or $x_1 = x_2, y_1 \leq y_2$ in the original coordinate system and $(x_1,y_1) \preceq (x_2,y_2)$ if $x_1 < x_2$ or $x_1 = x_2, y_2 \leq y_1$ in the new coordinate system.

\subsection{Discrete free energy}

Let $(x,\ell), (y,m) \in \Z_{\geq 1} \times \lb 1, n \rb$ where $x \leq y$ and $\ell \geq m$. Recall that $\Pi[(x,\ell) \rightarrow (y,m)]$ is the set of up-right paths from $(x,\ell)$ to $(y, m)$. There is an injective map from $\Pi[(x,\ell) \rightarrow (y,m)]$ to $\Z^{\ell -m}$, $\pi \rightarrow (t_{\ell}, \cdots, t_{m+1})$ where $t_i = \max \{t \in \lb x,y \rb| (t,i) \in \pi \}$. We will also set $
t_{\ell+1} =x, t_m = y$.

We define the discrete free energy of a discrete line ensemble as follows.

\begin{define}\label{def_free_energy}
    Let $r_1 \leq \cdots \leq r_n$ be $n$ weakly increasing positive integers. Let $f = (f_1, \cdots, f_n)$ be an $n$-tuple of functions where $f_i: \Z_{\geq r_i} \rightarrow \R$ for $i \in \lb 1, n \rb$. Assume that $f_i(r_i -1) = 0$ for all $i \in \llbracket 1, n \rrbracket$. We define $\{(x,y): y \in \lb 1, n \rb, x \in \Z_{\geq r_y}\}$ to be the domain of the discrete line ensemble $f$. For an single up-right path $\tau$ contained in the domain of $f$, define
    \begin{equation}
        f(\tau) := \sum_{j=m}^{\ell} f_j(t_j) - f_j(t_{j+1}-1).
    \end{equation}

    For any multipath $\pi = (\pi_1, \cdots, \pi_k)$ contained in the domain of $f$, define
    \begin{equation}
        f(\pi) := \sum_{i=1}^k f(\pi_i).
    \end{equation}

    Let $(U,V)$ be an endpoint pair contained in the domain of $f$. We define the free energy from $U$ to $V$ with respect to $f$ to be
    \begin{equation}
        f[U \rightarrow V] := \log \sum_{\pi \in \Pi[U \rightarrow V]} \exp( f(\pi)).
    \end{equation}
\end{define}

\begin{rmk}\label{rmk_equivalence_of_Z_and_f}
    Definition \ref{def_free_energy} is consistent with Definition \ref{def_polymer_partition}. Let $\{d_{i,j}\}_{i,j \in \Z}$ be a sequence of positive real number that represents the polymer weight. Define $f_i(x) = \log(\prod_{j=1}^x d_{j,n+1-i})$ or $d_{x,n+1-i} = \frac{\exp(f_i(x))}{\exp(f_{i}(x-1))}$. For any endpoint pair $(U,V)$ in the domain of $f$, let $(\tilde{U}, \tilde{V})$ be its image under the map $(x, i) \mapsto (x,n-i+1)$; then
    \begin{equation}
        \log Z[\tilde{U} \rightarrow \tilde{V}] = f[U \rightarrow V].
    \end{equation}
\end{rmk}

\subsection{Invariance of discrete free energy}
As noted in Remark \ref{rmk_equivalence_of_Z_and_f}, the discrete free energy of a discrete line ensemble, when defined in terms of the polymer weight, coincides with the logarithm of the polymer partition function. Thus, we will deduce the invariance of the free energy from the invariance of polymer partition function.

The polymer partition function can be viewed as the positive temperature analogue of the last passage percolation in the zero temperature setting. In the zero temperature case, Greene's theorem \cite{greene1974extension} tells us how to read off the last passage value that starts from the origin from the Robinson–Schensted–Knuth (RSK) correspondence \cite{robinson1938representations, knuth1970permutations,schensted1961longest}. In the positive temperature setting, with the $(\text{max}, +)$ semi-ring replaced by the $(+, \times)$ semi-ring, the geometric RSK correspondence arises as a natural analogue of the RSK correspondence. There are different proofs of the positive-temperature analogue of the Greene's theorem under the geometric RSK in the literature. The first can be found in the paper by Noumi and Yamada \cite{noumi2004tropical}. The proofs by Ivan Corwin and Konstantin Matveev can be found in \cite{corwin2020invariance}.

We will not work directly the discrete geometric Pitman transform operator $\W$ defined in \cite[Definition 2.3]{corwin2020invariance}. Instead, we introduce a variant of $\W$, which we denote by $W$. The operator $\W$ acts on any $n$-tuple of functions $D = (D_1, \cdots, D_n)$ where $D_1, \cdots, D_n: \Z_{\geq 1} \rightarrow (0, \infty)$. These functions should be viewed as the products of the polymer weight, i.e., $D_i(x) = \prod_{j=1}^x d_{j,n+1-i}$. Then $\W$ outputs an $n$-tuple of functions $\W D = (\W D_1, \cdots, \W D_n)$ where $\W D_i: \Z_{\geq i} \rightarrow (0, \infty)$ for all $i \in \lb 1 , n \rb$. The operator $W$ acts on any $n$-tuple of functions $f = (f_1, \cdots, f_n)$ where $f_1, \cdots, f_n: \Z_{\geq 1} \rightarrow \R$. Let $D = (\exp(f_1), \cdots, \exp(f_n))$. Then $W$ outputs an $n$-tuple of functions $Wf = (\log(\W D_1), \cdots,\log( \W D_n))$ for all $i \in \lb 1 , n \rb$.

The explicit definition for this operator $\W$ will not be needed in the rest of the paper. For readers interested in the precise construction, we refer to \cite{corwin2020invariance} for full details. For the remainder of this paper, we will only invoke how this operator acts on the inverse gamma random variables (introduced in Section 3) and the following adaptation of \cite[Theorem 2.4]{corwin2020invariance} in terms of the operator $W$:

\begin{thm}\label{thm_invariance}
     Let $U = \{(u_i, n)\}_{i \in \lb 1, k\rb}$ and $V = \{(v_i,1)\}_{i \in \lb 1, k \rb}$ be any endpoint pair and define $\Uparrow U := \{(u_i, n \wedge u_i)\}_{i \in \lb 1,k \rb}$. Let $f = (f_1, \cdots, f_n)$ be an $n$-tuple of functions where $f_i: \Z_{\geq 1} \rightarrow \R$ for $i \in \lb 1, n \rb$. Assume that $f_i(0) = 0$ for all $i \in \llbracket 1, n \rrbracket$. Then
    \begin{equation}
        f[U \rightarrow V] = Wf[\Uparrow U \rightarrow V].
    \end{equation}
    In particular, for any $N \in \Z_{\geq 1}$, $\ell \in \lb 1, n \wedge N \rb$, $U = \{(i,n)\}_{i \in \lb 1, \ell \rb}$, and $V = \{(j,1)\}_{j \in \lb N-\ell+1, N \rb}$, we have
    \begin{equation}
        f[U \rightarrow V] = \sum_{i=1}^{\ell} Wf_i(N).
    \end{equation}
\end{thm}

\begin{figure}[ht]
  \centering

\begin{tikzpicture}[>=Stealth, 
                    every node/.style={font=\footnotesize},
                    scale=1.0]

\begin{scope}
    \draw[step=1.0, gray!200, dotted] (0,0) grid (5,4);

    \node[left] at (0,0) {$f_5$};
    \node[left]       at (0,1) {$f_4$};
    \node[left]       at (0,2) {$f_3$};
    \node[left]       at (0,3) {$f_2$};
    \node[left]       at (0,4) {$f_1$};

    \draw[thick]
        (0,0) coordinate[pos=0.5](pi1) -- (0,1) 
        -- (1,1) -- (1,2) 
        -- (2,2) -- (2,3)
        -- (3,3) -- (3,4) coordinate[pos=0.5](pi2);

    \draw[thick]
        (2,0) coordinate[pos=0.5](pi3)-- (3,0) 
        -- (3,1) -- (3,2) 
        -- (4,2) -- (5,2)
        -- (5,3) -- (5,4) coordinate[pos=0.5](pi4);

    \node[above right] at (0,0) {$(1,5)$};
    \node[above right] at (pi2) {$(4,1)$};
    \node[above right] at (2,0) {$(2,5)$};
    \node[above right] at (pi4) {$(6,1)$};
    
    \fill (2,0) circle (2pt);
    \fill (0,0) circle (2pt);
    \fill (5,4) circle (2pt);
    \fill (3,4) circle (2pt);
  \draw[->] (0.9, -0.4) -- (0,-0.1) ;
    \draw[->] (1.1, -0.4) -- (2,-0.1) ;
    \node[below, yshift = -3mm] at (1,0) {$U$};
    \node[above, yshift = 3mm]      at (4,4) {$V$};
 \draw[->] (3.9, 4.4) -- (3,4.1) ;
    \draw[->] (4.1, 4.4) -- (5,4.1);

\end{scope}

\begin{scope}[xshift=7cm] 
    \draw[step=1.0, gray!200, dotted] (0,0) grid (5,4);



      \draw[line width=1mm, white] (0,0) -- (0,4);
      \draw[line width=3mm, white] (0,0) -- (4,0);
      \draw[line width=3mm, white] (1,0) -- (1,3);
      \draw[line width=3mm, white] (0,1) -- (3,1);
      \draw[line width=3mm, white] (0,2) -- (2,2);
      \draw[line width=3mm, white] (0,3) -- (1,3);
      \draw[line width=3mm, white] (2,0) -- (2,2);
      \draw[line width=3mm, white] (3,0) -- (3,1);

      \draw[thick]
        (0,4) -- (1,4)  -- (2,4)
        -- (3,4);
        \draw[thick]
        (2,2) -- (3,2)  -- (4,2)
        -- (4,3) -- (4,4) -- (5,4);

    \fill (2,2) circle (2pt);
    \fill (0,4) circle (2pt);
    \fill (3,4) circle (2pt);
    \fill (5,4) circle (2pt);

\node[below right] at (0,4) {$(1,1)$};
    \node[below right] at (2,2) {$(3,3)$};
    \node[below right] at (3,4) {$(4,1)$};
    \node[below right] at (5,4) {$(6,1)$};
    \node[left]       at (0,4) {$(Wf)_1$};
    \node[left]       at (1,3) {$(Wf)_2$};
    \node[left]       at (2,2) {$(Wf)_3$};
    \node[left]       at (3,1) {$(Wf)_4$};
    \node[left]       at (4,0) {$(Wf)_5$};
    \draw[->] (-0.4, 3.1) -- (0,3.9) ;
    \draw[->] (-0.4, 3) -- (2,2.1) ;
    \node[left, xshift = -3mm] at (0,3) {$\Uparrow U$};
    \node[above, yshift = 3mm]      at (4,4) {$V$};
 \draw[->] (3.9, 4.4) -- (3,4.1) ;
    \draw[->] (4.1, 4.4) -- (5,4.1);
\end{scope}
\end{tikzpicture}
  \caption{An example of Theorem \ref{thm_invariance} when $n = 5$ and $k = 2$. The $i$-th row of the grid is associated with function $f_i$ on the left and $(Wf)_i$ on the right. The illustrated paths are possible non-intersecting paths from $U$ to $V$ on the left and $\Uparrow U$ to $V$ on the right.}
  \label{fig:tikzgraph}
\end{figure}
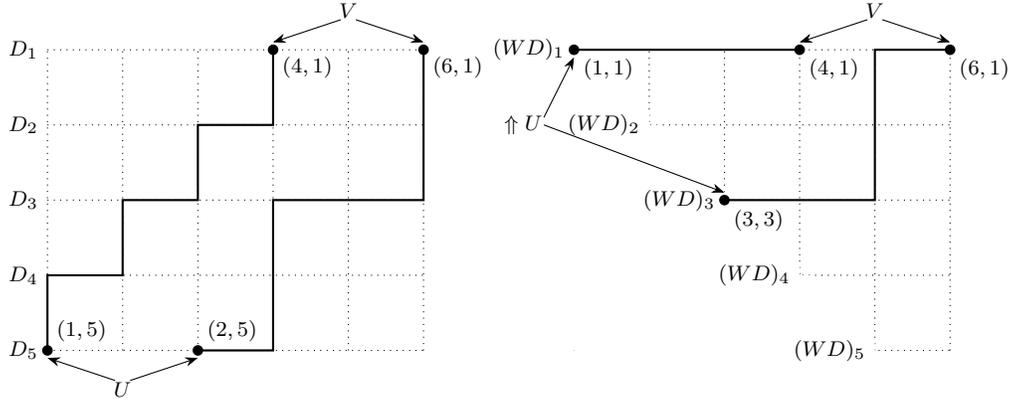

\subsection{Basic properties of discrete free energy}
The aim of this section is to establish a connection between the free energy $Wf[(x, n \wedge x) \rightarrow (y,k+1)]$, which depends solely on the information of the lower curves, and certain observables associated with the top $k+1$ curves of $Wf$ and $W(R_yf)$, to be defined below.

\begin{define}
    For any sequences of positive real numbers $d=\{d_{i,j}\}_{i \in \Z_{\geq 1}, j \in \lb 1,n \rb}$ and $z \in \Z_{\geq 1}$, we define the reverse environment $\{(R_zd)_{i,j}\}_{i \in \lb 1 , z \rb, j \in \lb 1, n \rb}$ by
    \begin{equation}
        (R_zd)_{i,j} = d_{z+1 - i, n + 1-j}
    \end{equation}
    for all $i \in \lb 1,z \rb$ and $j \in \lb 1, n \rb$.
\end{define}
Let $U= \{(x_i, \ell_i)\}_{i \in \lb 1, k \rb}$ and $V = \{(y_i,m_i)\}_{i \in \lb 1, k \rb}$ be an endpoint pair with $U,V \subset \lb 1, z \rb \times \lb 1, n \rb$. Let 
\begin{equation}
R_z U := \{(z + 1-x_{i}, n+1 - \ell_{i})\}_{{i \in \lb 1, k \rb}} , \quad
R_z V := \{(z+1-y_{i}, n+1 - m_{i}\}_{i \in \lb 1, k \rb}.   
\end{equation}
Lastly, let $\tilde{U}= \{(x_i, n-\ell_i+1)\}_{i \in \lb 1, k \rb}$, $\tilde{V} = \{(y_i,n-m_i+1)\}_{i \in \lb 1, k \rb}$, $R_z \tilde{U} = \{(z + 1-x_{i}, \ell_{i})\}_{{i \in \lb 1, k \rb}}$, and $R_z \tilde{V} = \{(z + 1-y_{i}, m_{i})\}_{{i \in \lb 1, k \rb}}$.

\begin{lemma}\label{lemma_reverse_map}
Under the setting above, let $Z[\tilde{U} \rightarrow \tilde{V}]$ to be the polymer partition function with respect to $d=\{d_{i,j}\}_{i \in \Z_{\geq 1}, j \in \lb 1,n \rb}$ and let $(R_zZ)[\tilde{U} \rightarrow \tilde{V}]$ to be the polymer partition function with respect to $\{(R_zd)_{i,j}\}_{i \in \lb 1 , z \rb, j \in \lb 1, n \rb}$. Then,
    \begin{equation}\label{2}
     Z[\tilde{U} \rightarrow \tilde{V}] = (R_zZ)[R_z \tilde{U} \rightarrow R_z \tilde{V}].  
    \end{equation}
    Let $f = (f_1, \cdots, f_n)$ be an $n$-tuple of functions where $f_i: \lb 1 ,z \rb \rightarrow \R$ for $i \in \lb 1, n \rb$. Assume that $f_i(0) = 0$ for all $i \in \llbracket 1, n \rrbracket$. Let $d_{i,j} = \frac{\exp(f_{n+1-j}(i))}{\exp(f_{n+1-j}(i-1))}$ for all $i,j \in \lb 1, z \rb \times \lb 1, n \rb$. We define $R_zf=((R_zf)_1, \cdots, (R_zf)_n): \lb 0, z \rb^n \rightarrow \R^n$ by $(R_zf)_i(x) = \log \prod_{j=1}^x (R_zd)_{j,n+1-i}$ and $(R_zf)_i(0) = 0$ for all $i \in \lb 1, n \rb$ and $x \in \lb 1, z \rb$. Then we have
    \begin{equation}\label{3}
        f[U \rightarrow V] = (R_zf)[R_z U \rightarrow R_z V].
    \end{equation}
\end{lemma}

\begin{proof}
    Since the reverse map $R_z$ can simply be viewed as rotating the grid by $180$ degrees, Equation \ref{2} follows directly. Equation \ref{3} follows from Remark \ref{rmk_equivalence_of_Z_and_f}.
\end{proof}

We also need a reverse version of the free energy. In Definition \ref{def_free_energy}, the free energy is expressed in terms of up-right paths. In the reverse setting, the paths travel in down-right direction. However, the notion of ``allowable path'' is slightly different. Thus, instead of working with lattice paths, it is more convenient to describe the configuration in terms of the turning points $t_i$, allowing for a cleaner formulation of Lemma \ref{lemma_searrow}.

\begin{define}\label{def_reverse_free_energy}
    For $z,w \in \Z_{\geq 1}$, $z \leq w$ and $k+1 \in \lb 1 , n \rb$, let
    \begin{equation}
    \begin{split}
        f[(z,1) {\searrow} (w,k+1)] :=  \quad   -\log \sum_{z < t_1 < \cdots < t_k \leq w} \exp\left( f_1(z) - f_{k+1}(w) + \sum_{i=1}^k \biggl( f_{i+1}(t_i) - f_i(t_i - 1) \biggr) \right).
    \end{split}
\end{equation}
\end{define}

With the reverse free energy defined, we can now state the precise relationship between the free energy on the lower curves (i.e., the $(k+1)$-st through $n$-th curves) and the top $k+1$ curves.
\begin{lemma}\label{key1}
    For $x,y \in \Z_{\geq 1}$, $k \in \lb 1, n-1 \rb$, and $k < x$, $x < y-k+1$, we have
    \begin{equation}\label{equality}
    Wf[(x,n \wedge x) \rightarrow (y,k+1)]  = (Wf)_{k+1}(y) - W(R_y f)[(y-x+1,1) {\searrow} (y,k+1)]).
    \end{equation}
\end{lemma}

The remainder of this section is dedicated to proving Lemma \ref{key1}, which is analogous to \cite[Lemma 5.3]{dauvergne2022directed} and \cite[Proposition 3.5]{wu2023kpz}. For notational convenience, let $U_{n,k} = \{(i,n): i \in \lb 1,k \rb\}$ and $\Uparrow U_{n,k} = \{(i, n \wedge i): i \in \lb 1, k \rb \}$. Here $U_{n,k}$ denotes the bottom horizontal segment consisting of the first $k$ points from the left and $\Uparrow U_{n,k}$ represents its lifted version after the application of the Pitman transform. Let $H_k(y)= \{(i,1): i \in \lb y-k+1,y \rb \}$. Let $V_k(y) = \{(y,i):i \in \lb n-k+1, n\rb \}$. The set $H_k(y)$ denotes the top horizontal segment of $k$ consecutive points with rightmost point $(y,1)$. The set $V_k(y)$ denotes the top vertical segment of $k$ consecutive points with topmost point $(y,1)$. Theorem \ref{thm_invariance} implies a saturation phenomenon for paths when the set of starting points include $U_{n,k}$. We leverage this saturation behavior to establish the following two lemmas, which collectively implies Lemma \ref{key1}.

\begin{lemma}\label{lemma_concentration}
    For $x,y \in \Z_{\geq 1}$, $k \in \lb 1, n-1 \rb$, and $k < x < y$, we have that
    \begin{equation}
    (Wf)[(x,n \wedge x) \rightarrow (y,k+1)] =   f[\{U_{n,k},(x,n)\} \rightarrow H_{k+1}(y)] - \sum_{i=1}^k (Wf)_i(y).
    \end{equation}
\end{lemma}

\begin{lemma}\label{lemma_searrow}
    For $x,y \in \Z_{\geq 1}$, $k \in \lb 1, n-1 \rb$, and $k < x < y-k+1$, we have
    \begin{equation}
    f[U_{n,k+1} \rightarrow \{(x,1), H_k(y)\}] = f[U_{n,k+1} \rightarrow H_{k+1}(y)] - Wf [(x,1) {\searrow} (y,k+1)]
    \end{equation}
\end{lemma}

The proof of Lemma \ref{key1} follows by noting the identity
\[f[\{U_{n,k},(x,n)\} \rightarrow H_{k+1}(y)] = R_yf[U_{n,k+1} \rightarrow \{(y-x+1,1), H_k(y)\}].\]
We then apply Lemma \ref{lemma_searrow} to $R_yf$ and observe that 
\[
R_yf[U_{n,k+1} \rightarrow H_{k+1}(y)] = f[U_{n,k+1} \rightarrow H_{k+1}(y)] = \sum_{i=1}^{k+1} (Wf)_i(y). 
\]

Now we proceed to prove Lemma \ref{lemma_concentration} and Lemma \ref{lemma_searrow}.

\begin{proof}[Proof of Lemma \ref{lemma_concentration}]
Let us apply Theorem \ref{thm_invariance}. 
\[
    \begin{split}
        f[\{U_{n,k},(x,n)\} \rightarrow H_{k+1}(y)]
        =Wf[\{ \Uparrow U_{n,k}, (x,n \wedge x)\} \rightarrow H_{k+1}(y)]).
    \end{split} 
\]
Consider any multipath $\pi = (\pi_1, \cdots, \pi_{k+1})$ such that $\pi_i$ connects $(i, i)$ to $(y-k-1+i,1)$ for $i \in \lb 1 , k \rb$ and $\pi_{k+1}$ that connects $(x, n \wedge x)$ to $(y, 1)$. Notice that the first $k$ rows are saturated by paths and the remaining degrees of freedom lie in the part of $\pi_{k+1}$ that connects $(x, n \wedge x)$ to $(y, k+1)$ as shown in Figure \ref{fig2}.

\begin{figure}[ht]
  \centering

\begin{tikzpicture}[>=Stealth, 
                    every node/.style={font=\footnotesize},
                    scale=1.0]

\begin{scope} 
    \draw[step=1.0, gray!200, dotted] (0,0) grid (5,4);



      \draw[line width=1mm, white] (0,0) -- (0,4);
      \draw[line width=3mm, white] (0,0) -- (4,0);
      \draw[line width=3mm, white] (1,0) -- (1,3);
      \draw[line width=3mm, white] (0,1) -- (3,1);
      \draw[line width=3mm, white] (0,2) -- (2,2);
      \draw[line width=3mm, white] (0,3) -- (1,3);
      \draw[line width=3mm, white] (2,0) -- (2,2);
      \draw[line width=3mm, white] (3,0) -- (3,1);

      \draw[thick]
        (0,4) -- (1,4)  -- (2,4)
        -- (3,4);
        \draw[thick]
        (1,3) -- (2,3)  -- (3,3)
        -- (4,3) -- (4,4);
        \draw[thick]
        (4,0) -- (4,1)  -- (5,1)
        -- (5,2) -- (5,3) -- (5,4);

    \fill (1,3) circle (2pt);
    \fill (1,4) circle (2pt);
    \fill (2,4) circle (2pt);
    \fill (3,4) circle (2pt);
    \fill (2,3) circle (2pt);
    \fill (3,3) circle (2pt);
    \fill (4,3) circle (2pt);
    \fill (5,3) circle (2pt);
    \fill (4,1) circle (2pt);
    \fill (5,1) circle (2pt);
    \fill (0,4) circle (2pt);
    \fill (5,2) circle (2pt);
    \fill (4,4) circle (2pt);
    \fill (5,4) circle (2pt);
    \fill (4,0) circle (2pt);

\node[below right] at (0,4) {$(1,1)$};
    \node[below right] at (1,3) {$(2,2)$};
    \node[below right] at (3,4) {$(4,1)$};
    \node[below right] at (4,4) {$(5,1)$};
    \node[below right] at (5,4) {$(6,1)$};
    \node[below right] at (4,0) {$(5,5)$};

\end{scope}

\begin{scope}[xshift=7cm] 
    \draw[step=1.0, gray!200, dotted] (0,0) grid (5,4);



      \draw[line width=1mm, white] (0,0) -- (0,4);
      \draw[line width=3mm, white] (0,0) -- (4,0);
      \draw[line width=3mm, white] (1,0) -- (1,3);
      \draw[line width=3mm, white] (0,1) -- (3,1);
      \draw[line width=3mm, white] (0,2) -- (2,2);
      \draw[line width=3mm, white] (0,3) -- (1,3);
      \draw[line width=3mm, white] (2,0) -- (2,2);
      \draw[line width=3mm, white] (3,0) -- (3,1);

      \draw[thick]
        (0,4) -- (1,4)  -- (2,4)
        -- (3,4) -- (4,4);
        \draw[thick]
        (1,3) -- (2,3)  -- (3,3) 
        -- (4,3) -- (5,3) -- (5,4) ;
        \draw[thick]
        (4,0) -- (4,1)  -- (5,1)
        -- (5,2);

    \fill (1,3) circle (2pt);
    \fill (1,4) circle (2pt);
    \fill (2,4) circle (2pt);
    \fill (3,4) circle (2pt);
    \fill (2,3) circle (2pt);
    \fill (3,3) circle (2pt);
    \fill (4,3) circle (2pt);
    \fill (5,3) circle (2pt);
    \fill (4,1) circle (2pt);
    \fill (5,1) circle (2pt);
    \fill (0,4) circle (2pt);
    \fill (5,2) circle (2pt);
    \fill (4,4) circle (2pt);
    \fill (5,4) circle (2pt);
    \fill (4,0) circle (2pt);

\node[below right] at (0,4) {$(1,1)$};
    \node[below right] at (1,3) {$(2,2)$};
    \node[below right] at (5,2) {$(6,3)$};
    \node[below right] at (4,4) {$(5,1)$};
    \node[below right] at (5,4) {$(6,1)$};
    \node[below right] at (4,0) {$(5,5)$};

\end{scope}
\end{tikzpicture}
  \caption{An example of a multipath reorganization when $k = 2$, $y = 6$, and $x = 5$. The diagram on the right illustrates a reorganization of the multipath in the left diagram. The first $k$ rows 
are fully saturated by paths and the only remaining degree of freedom lies in the segment 
of $\pi_{k+1}$ connecting $(x, n \wedge x)$ to $(y, k+1)$. Importantly, the free energy 
is unchanged under this reorganization, as what matters is the collection of points 
visited by the multipath.}

  \label{fig2}
\end{figure}

Let $\pi^{\text{Tr}}_{k+1}$ denote the truncated path $\pi_{k+1}$ from $(x, n \wedge x)$ to $(y,k+1)$. Thus, the free energy $W f[\{ \Uparrow U_{n,k}, (x,n \wedge x)\} \rightarrow H_{k+1}(y)]$ factors nicely:
\[
    \begin{split}
        Wf[\{ \Uparrow U_{n,k}, (x,n \wedge x)\} \rightarrow H_{k+1}(y)])
        &= \log \sum_{\pi \in \Pi[\{ \Uparrow U_{n,k}, (x,n \wedge x)\} \rightarrow H_{k+1}(y)]}  \exp(Wf(\pi))\\
        &=  \log \sum_{\pi \in \Pi[\{ \Uparrow U_{n,k}, (x,n \wedge x)\} \rightarrow H_{k+1}(y)]}  \exp(Wf(\pi^{\text{Tr}}_{k+1}) + \sum_{i=1}^k Wf(y)))\\
        & = \sum_{i=1}^k (Wf)_i(y)+  Wf[(x,n \wedge x) \rightarrow (y,k+1)].
    \end{split}
\]
\end{proof}
\begin{proof}[Proof of Lemma \ref{lemma_searrow}]
Again, we apply Theorem \ref{thm_invariance}:
\[
\begin{split}
    f[U_{n,k+1} \rightarrow \{(x,1), H_k(y)\}]= Wf[\Uparrow U_{n,k+1} \rightarrow \{(x,1), H_k(y)\}].
\end{split}
\]
Consider any multipath $\pi = (\pi_1, \cdots, \pi_{k+1})$ such that $\pi_1$ connects $(1,1)$ to $(x,1)$ and $\pi_i$ connects $(i,i)$ to $(y-k+i-1,1)$ for all $i \in \lb 2, k+1 \rb$. Due to the non-intersecting property, we know that $\pi_i$ must pass through $(y-k+i-1, i-1)$ then move vertically upward. Thus, there exists a bijection between $\Pi[\Uparrow U_{n,k+1} \rightarrow \{(x,1), H_k(y)\}]$ and $\{(t_0, t_1, \cdots, t_k) \in \lb x, y \rb^{k+1}: x = t_0 < t_1 < t_2 < \cdots < t_k \leq y\}$ where $\pi_{i+1}$ travels along $(i+1)$-th row until $(t_i, i+1)$, where it turns upward, for all $i \in \lb 1, k \rb$. Then,
\[
\begin{split}
    &Wf[\Uparrow U_{n,k+1} \rightarrow \{(x,1), H_k(y)\}] \\
    &=\log \sum_{\pi \in \Pi[\Uparrow U_{n,k+1} \rightarrow \{(x,1),H_{k}(y)\}]}  \exp(Wf(\pi))\\
    &=\log \sum_{x = t_0 < t_1 < \cdots < t_k \leq y}  \exp\left( Wf_{k+1}(t_k) + \sum_{i=1}^k Wf_i(y) - Wf_i(t_i-1) + Wf_i(t_{i-1})\right) \\
    &=\sum_{i=1}^{k+1} Wf_i(y) + \log \sum_{x < t_1 < \cdots < t_k \leq y}  \exp \left( Wf_1(x) -Wf_{k+1}(y) + \sum_{i=1}^k  Wf_{i+1}(t_i)- Wf_i(t_i-1)\right) \\  
    &= f[U_{n,k+1} \rightarrow H_{k+1}(y)] - Wf[(x,1) \searrow (y,k+1)].
    \end{split}
\]
The last line follows from Theorem \ref{thm_invariance} and Definition \ref{def_reverse_free_energy}.
\end{proof}

\subsection{Variational formula and monotonicity}
While the positive temperature polymer models do not satisfy a direct metric composition law for their free energy, they admit the following variational formula. 
\begin{lemma}\label{Bisection}
    For any $m,k,\ell \in \lb 1,n \rb$, $m \leq k < \ell$, $x, y \in \Z_{\geq 1}$ and $x < y$, we have 
    \begin{equation}\label{equation_variational_f}
    \exp(f[(x,\ell) \rightarrow (y,m)]) = \sum_{i=x}^y \exp(f[(x,\ell) \rightarrow (i,k+1)] + f[(i,k) \rightarrow (y,m)]).
    \end{equation}
\end{lemma}

\begin{proof}
    \[
    \begin{split}
        \exp(f[(x,\ell) \rightarrow (y,m)]) &=
        \sum_{\pi \in \Pi[(x,\ell) \rightarrow (y,m)]} \exp(f(\pi))\\
        &= \sum_{i=x}^y \sum_{\pi \in \Pi[(x,\ell) \rightarrow (i,k+1)]}\sum_{\tilde{\pi} \in \Pi[(i,k) \rightarrow  (y,m)]} \exp(f(\pi))\exp(f(\tilde{\pi}))\\
        &= \sum_{i=x}^y \exp(f[(x,\ell) \rightarrow (i,k+1)] + f[(i,k) \rightarrow(y,m)])
    \end{split}
    \]
\end{proof}
The variational formula thus provides a crucial mechanism for distinguishing the asymptotically dominant contributions: in the convergence to the Airy sheet, only the contribution from the top part of the free energy survives in the limit, while the lower part becomes negligible. We now establish a monotonicity property of the free energy, which will later be used to show that the contribution from the lower part is negligible in the scaling limit. Intuitively, this monotonicity implies that path coalescence is energetically more favorable. The zero-temperature analogue of such monotonicity follows directly from definition while the positive temperature semi-discrete analogue can be found in \cite[Lemma 2.4]{wu2023kpz}.

\begin{lemma}\label{lemma_monotonicity_freeenergy}
    For any $m,\ell \in \lb 1, n \rb$ where $m \leq \ell$, $x_1,x_2 \in \Z_{\geq 1}$ where $x_1 \leq x_2$, the following inequality holds for any $y_1, y_2 \in \Z_{\geq x_2}$ where $y_1 \leq y_2$:
    \begin{equation}
    f[(x_2, \ell) \rightarrow (y_1,m)] - f[(x_1, \ell) \rightarrow (y_1,m)] \leq  f[(x_2, \ell) \rightarrow (y_2,m)] - f[(x_1, \ell) \rightarrow (y_2,m)].
    \end{equation}
\end{lemma}

\begin{proof}
    The proof is by induction on $\ell -m$. When $\ell = m$, the lemma follows trivially since $f[(x,m) \rightarrow (y,m)] = f_m(y) - f_m(x-1)$. Fix $\ell$ and assume that the inequality holds for $m+1$, and we will prove it for $m$. It suffices to consider $y_1 = y$ and $y_2 = y+1$. Let us rewrite the free energy using the variational formula (\ref{equation_variational_f}):
    \begin{equation}
        \exp(f[(x,\ell) \rightarrow (y,m)]) = \sum_{i=x}^y \exp(f[(x,\ell) \rightarrow (i,m+1)]+f_m(y)-f_m(i-1)).
    \end{equation}
    For $i,j \in \{1,2\}$, let
    \begin{equation}
        A(x_i,x_j) = \sum_{i=x_i}^y \exp(f[(x_j,\ell) \rightarrow{} (i,m+1)] - f_m(i-1))
    \end{equation}
    \begin{equation}
        B(x_i) = \exp(f[(x_i, \ell) \rightarrow (y+1,m+1)] - f_m(y)).
    \end{equation}
    By the variational formula, the inequality for $m$ is just
    \begin{equation}
        \frac{A(x_2,x_2)+B(x_2)}{A(x_1,x_1)+B(x_1)} \geq \frac{A(x_2,x_2)}{A(x_1,x_1)},
    \end{equation}
    which is equivalent to
    \begin{equation}
    \frac{A(x_1,x_1)}{A(x_2,x_2)} \geq \frac{B(x_1)}{B(x_2)}.   
    \end{equation}
    This follows by the induction hypothesis and the fact that every summand in $A(x_i,x_j)$ is non-negative. 
    \[
    \begin{split}
         A(x_2,x_2)
         &= \sum_{i=x_2}^{y} \exp(f[(x_2, \ell) \rightarrow (i,m+1)] - f[(x_1, \ell) \rightarrow (i,m+1)]+ f[(x_1, \ell) \rightarrow (i,m+1)]- f_m(i-1)) \\
         &\leq \exp(f[(x_2, \ell) \rightarrow (y+1,m+1)] -f[(x_1, \ell) \rightarrow (y+1,m+1)]) A(x_2,x_1)\\
         &\leq \exp(f[(x_2, \ell) \rightarrow (y+1,m+1)] -f[(x_1, \ell) \rightarrow (y+1,m+1)]) A(x_1, x_1).
    \end{split}
\]
\end{proof}

Similarly, we have the following quadrangle inequality for the partition function of the polymer model, which follows from Lemma \ref{lemma_monotonicity_freeenergy} by exponentiating both sides. 

\begin{lemma}\label{lemma_monotonicity_partitionfunction}
    For any $x_1, x_2, y_1, y_2 \in \Z$ where $x_1 \leq x_2 \leq y_1 \leq y_2$, we have
    \begin{equation}
    Z[(x_1, 1) \rightarrow 
    (y_2,n)]\cdot Z[(x_2, 1) \rightarrow 
    (y_1,n)] \leq Z[(x_2, 1) \rightarrow 
    (y_2,n)] \cdot Z[(x_1, 1) \rightarrow 
    (y_1,n)]
    \end{equation}
\end{lemma}

\section{From free energy to last passage percolation}\label{section3}
For this section, we fix a positive integer $N$ and will address the regime $N \rightarrow \infty$ in the next section. Recall that we have defined the polymer partition functions on $\Z^2$ with respect to $d = \{d_{i,j}\}_{i,j \in \Z}$ according to Definition \ref{def_polymer_partition}. On the domain $\lb 1, 3N \rb \times \lb 1 , 2N \rb$, we define $f^N = (f_1^N, \cdots, f_{2N}^N)$ by setting $f_i^N(x) = \log(\prod_{j=1}^x d_{j,2N+1-i})$ for all $(x,i) \in \lb 1,3N \rb \times \lb 1, 2N \rb$. The choice of the domain dimensions $2N$ and $3N$ ensures that when we apply the operator $W$ introduced in Theorem \ref{thm_invariance} to $f$, the image $Wf^N = (Wf^N_1, \cdots, Wf^N_{2N})$ is well defined on the domain $\lb N,3N \rb \times \lb 1, N \rb$.

\subsection{Log-gamma line ensembles to parabolic Airy line ensemble} 
In this section, we define the log-gamma line ensemble following the conventions of \cite{dimitrov2021tightness}. The log-gamma line ensemble arises naturally from the study of the log-gamma polymer, and can be viewed as a multi-layer extension of the polymer’s free energy profile. Its construction is rooted in the connection between the log-gamma polymer and the Whittaker processes described by \cite{CorwinOConnellSeppalainenZygouras2014}. In contrast to the Airy line ensemble governed by the Brownian Gibbs property, the log-gamma line ensemble exhibits a more intricate local Gibbs structure. Subsequent works, such as \cite{dimitrov2021tightness}, have formalized this ensemble within the framework of Gibbsian line ensembles, proving tightness and establishing conditions under which scaled versions converge to the universal Airy line ensemble, as rigorously demonstrated in \cite{aggarwal2023strong}. To proceed with a precise formulation of the log-gamma line ensemble, we introduce a few important functions.
\begin{define}\label{def_psi}
    We use \(\Psi(x)\) to denote the digamma function, the logarithmic derivative of the gamma function. Define the function

\begin{equation}
g_{\theta}(z) = \frac{\sum_{n=0}^{\infty} \frac{1}{(n+\theta-z)^2}}{\sum_{n=0}^{\infty} \frac{1}{(n+z)^2}} = \frac{\Psi'(\theta - z)}{\Psi'(z)},
\end{equation}
which is a smooth and strictly increasing bijection from \((0, \theta)\) to \((0, \infty)\). Its inverse, denoted \(g_{\theta}^{-1} : (0, \infty) \to (0, \theta)\), is also smooth and strictly increasing. Using this, we define

\begin{equation}
h_{\theta}(x) = x \cdot \Psi(g_{\theta}^{-1}(x)) + \Psi(\theta - g_{\theta}^{-1}(x)),
\end{equation}
a smooth function on $(0, \infty)$.
Let $p = -h_\theta'(1) = -\frac{h_\theta(1)}{2}$ and $\sigma_p = \left[\Psi'(\theta/2)\right]^{-1/2}$.

Lastly, we define the function
\begin{equation}
    d_\theta(x) = \Bigg(\sum_{n=0}^\infty \frac{x}{(n+g_\theta^{-1}(x))^3} + \sum_{n=0}^\infty \frac{x}{(n+ \theta - g_\theta^{-1}(x))^3}\Bigg)^{1/3}.
\end{equation}
\end{define}

\begin{define}
For $i \in \lb 1, N \rb$ and $j \in \lb-N,N \rb$, we define  
\begin{equation}
L_i^N(j) = W f_i^N(2N+j) + 2Nh_\theta(1).
\end{equation}
We extend the domain of $L_i^N$ to $[-N,N]$ by linear interpolation between integer points. 

We define the log-gamma line ensemble $\A^N = (\A_1^N, \A_2^N, \cdots)$ as follows. For $i \in \lb 1, N \rb$, set
\begin{equation}
\A^N_i(s)=
\begin{cases}
\sigma_pN^{-1/3}(L_i^N(-\frac{1}{2}N)+(p/2)N) & \text{if } s \leq -\frac{1}{2}N^{1/3}\\
\sigma_pN^{-1/3}(L_i^N(sN^{2/3})-psN^{2/3}) & \text{if } s \in [-\frac{1}{2}N^{1/3},\frac{1}{2}N^{1/3}]\\
\sigma_pN^{-1/3}(L_i^N(\frac{1}{2}N)-(p/2)N) & \text{if } s \geq \frac{1}{2}N^{1/3}.
\end{cases}
\end{equation}
For $i \geq N+1$, we define $\A^N_i \equiv 0$. 
\end{define}

With the log-gamma line ensemble $\A^N$ now defined, we are ready to state the convergence result established in \cite[Corollary 25.2]{aggarwal2023strong}.
\begin{thm}\label{thm_airy_line_ensemble}
    Let $q = 2^{-5/6}\sigma_p d_\theta(1)^{-1}$. The log-gamma line ensemble $\A^N = (\A^N_1, \A^N_2, \cdots)$ converges to rescaled parabolic Airy line ensemble $2^{-1/2}\A^q$ in distribution uniformly on compact subsets of $\Z_{\geq 1} \times \R$.
\end{thm}

\subsection{Log-gamma sheet and monotonicity}
We begin by introducing the appropriate scaling operations. Let $\overline{x}_N = \lfloor N^{2/3}xq^{-2} \rfloor +1$, $\widehat{y}_N = \lfloor N^{2/3}yq^{-2} \rfloor + 2N$.

\begin{define}[Log-gamma sheet]
 For $x,y \in \R$, we define the log-gamma sheet by
    \begin{equation}
h^N(x,y) = 2^{-1/2}q\sigma_pN^{-1/3}[\log Z((\overline{x}_N,1) \rightarrow (\widehat{y}_N,2N))-p(\widehat{y}_N - \overline{x}_N + 2N)].
\end{equation}
Given Remark \ref{rmk_equivalence_of_Z_and_f} and Theorem \ref{thm_invariance}, we can also express the log-gamma sheet in the following alternative form when $x > 0$, which will be the main focus in our analysis of convergence to the Airy sheet.
\begin{equation}
    h^N(x,y) = 2^{-1/2}q\sigma_pN^{-1/3}[Wf^N[(\overline{x}_N,\overline{x}_N \wedge 2N) \rightarrow (\widehat{y}_N,1)] -p(\widehat{y}_N-\overline{x}_N + 2N)].
\end{equation}
We also define two component functions of the log-gamma sheet for some $k \in \lb 1 , N-1 \rb$, $x > 0$, and $z,y \in \R$ such that $z \leq y$:
\begin{equation}
    F_k^N(x,z) = 2^{-1/2}q\sigma_pN^{-1/3}\left[Wf^N[(\overline{x}_N, \overline{x}_N \wedge 2N) \rightarrow (\widehat{z}_N,k+1)] - (Wf^N)_{k+1}(\widehat{z}_N)+ p(\overline{x}_N - N^{2/3}k) \right],
\end{equation}
\begin{equation}
    G_k^N(x,z) = 2^{-1/2}q\sigma_pN^{-1/3}\left[Wf^N[(\widehat{z}_N, k) \rightarrow (\widehat{y}_N,1)] + (Wf^N)_{k+1}(\widehat{z}_N)-p(\widehat{y}_N - N^{2/3}k+2N)\right].
\end{equation}
\end{define}

To apply the variational formula (\ref{equation_variational_f}), we need to work with the unscaled version of the log-gamma sheet and the component functions.

\begin{define}
    For $x,y,z \in \Z_{\geq 1}$, we define the unscaled log-gamma sheet and its two component functions as follows:
    \begin{equation}
    {\h}^N(x,y) = Wf^N[(x,x \wedge 2N) \rightarrow (y,1)] -p(y-x + 2N)
    \end{equation}
    \begin{equation}
    \F_k^N(x,z) = Wf^N[(x, x \wedge 2N) \rightarrow (z,k+1)] - (Wf^N)_{k+1}(z)+ p(x - N^{2/3}k)
    \end{equation}
    \begin{equation}
    \G_k^N(z,y) = Wf^N[(z, k) \rightarrow (y,1)] + (Wf^N)_{k+1}(z)-p(y - N^{2/3}k+2N).
    \end{equation}
\end{define}

The following definition connects the unscaled log-gamma sheet $\h^N$ and its two component functions $\F_k^N$ and $\G_k^N$ through a naturally associated probability measure:
\begin{define}
    For $x,y,z \in \Z_{\geq 1}$ and $x\leq z \leq y$, we denote by $\mu^N_{k,x,y}$ the random probability measure on the discrete set $\lb x, y \rb$ as follows:
    \begin{equation}
    \mu^N_{k,x,y}(z) = \exp(-\h^N(x,y) + \F_k^N(x,z) + \G^N_k(z,y)).
    \end{equation}
\end{define}
The fact that $\mu_{k,x,y}^N$ defines a probability measure follows from Lemma \ref{Bisection}. 

We also set its upper and lower cumulative distribution function to be
\begin{equation}
A_k^N (x,y;z) = \sum_{i = z}^{y} \mu^N_{k,x,y}(i), \quad 
B_k^N (x,y;z) = \sum_{i = x}^{z} \mu^N_{k,x,y}(i).
\end{equation}

We have the following monotonicity corollary of the log-gamma sheet and its component functions from Lemma \ref{lemma_monotonicity_freeenergy} and Lemma \ref{lemma_monotonicity_partitionfunction}.
\begin{cor}\label{cor_monotonicity}
    Given $x_1, x_2,y_1,y_2 \in \Z_{\geq 1}$, $x_1 \leq x_2$, and $y_1 \leq y_2$, $\F^N_k(x_2,z) - \F^N_k(x_1,z)$ is increasing in $z \in \Z_{\geq x_2}$ and $\G_k^N(z,y_2) - \G^N_k(z,y_1)$ is increasing in $z \in \Z_{\leq y_1}$. 
\end{cor}

\begin{cor}\label{quadrangle}
    Given $x_1,x_2, y_1, y_2 \in \R$, $x_1 \leq x_2$, $y_1 \leq y_2$, and $\overline{x}_{2_N} \leq \widehat{y}_{1_N}$,
    \begin{equation}
    h^N(x_1,y_1) + h^N(x_2,y_2) \geq h^N(x_1,y_2) + h^N(x_2,y_1).
    \end{equation}
\end{cor}

We now use Corollary \ref{cor_monotonicity} above to derive a series of inequalities. The key idea is to bound the difference between $\h^N(x,y_2) - \h^N(x,y_1)$ and $\G_k^N(x,y_2) - \G_k^N(x,y_1)$ in terms of the measure $\mu^N_{k}$. We then control the behavior of $\mu^N_k$ through the component function $\F_k^N$. Lastly, we will pass these inequalities to the log-gamma sheet and its two component function.

\begin{lemma}
Let $x_1, x_2, y_1, y_2, z \in \Z_{\geq 1}$ and $x_1 \leq x_2 \leq z \leq y_1 \leq y_2$. We have 
    \begin{equation}\label{eq_314}
    \F_k^N(x_2,z) - \F_k^N(x_1,z) \leq \h^N(x_2,y_1) - \h^N(x_1, y_1) - \log A_k^N (x_1,y_1;z)
    \end{equation}
    \begin{equation}\label{eq_315}
    \F_k^N(x_2,z) - \F_k^N(x_1,z) \geq \h^N(x_2,y_1) - \h^N(x_1, y_1) + \log B_k^N (x_2,y_1;z)
    \end{equation}
    \begin{equation}\label{eq_316}
    \G_k^N(z,y_2) - \G_k^N(z,y_1) \leq \h^N(x_2,y_2) - \h^N(x_2,y_1) - \log A_k^N(x_2,y_1;z)
    \end{equation}
    \begin{equation}\label{eq_317}
    \G_k^N(z,y_2) - \G_k^N(z,y_1) \geq \h^N(x_2,y_2) - \h^N(x_2,y_1) + \log B_k^N(x_2,y_2;z)
    \end{equation}
\end{lemma}

\begin{proof}
We will only prove equation (\ref{eq_314}) and the proofs for the rest follow analogously.
    \[
    \begin{split}
        \exp(\h^N(x_2,y) - \h^N(x_1, y))
        &= \sum_{w =  x_2}^{y} \exp(\F_k^N(x_2, w) + \G_k^N (w,y) - \h^N(x_1,y)+ \F_k^N(x_1,w) - \F_k^N(x_1,w))\\
       &\geq \sum_{w = z}^{y} \exp(\F_k^N(x_2, w) + \G_k^N (w,y) - \h^N(x_1,y)+ \F_k^N(x_1,w) - \F_k^N(x_1,w))\\
       &\geq \sum_{w = z}^{y} \exp(\F_k^N(x_2, z) + \G_k^N (w,y) - \h^N(x_1,y)+ \F_k^N(x_1,w) - \F_k^N(x_1,z))\\
       &=  \exp(\F_k^N(x_2, z)- \F_k^N(x_1,z)) A_k^N(x_1,y;z)
    \end{split}
    \]
    where we used Corollary \ref{cor_monotonicity} in the last inequality. 
\end{proof}

\begin{cor}\label{h_S}
For all large $N \in \N$ such that $2^{-1/2}q \sigma_pN^{-1/3} \leq 1$,
we have
    \begin{equation}
    F_k^N(x_2,z) -F_k^N(x_1,z) \leq h^N(x_2,y_1) - h^N(x_1, y_1) - \log A_k^N (\overline{x}_{1,N},\widehat{y}_{1,N};\widehat{z}_N)
    \end{equation}
    \begin{equation}
    F_k^N(x_2,z) - F_k^N(x_1,z) \geq h^N(x_2,y_1) - h^N(x_1, y_1) + \log B_k^N (\overline{x}_{2,N},\widehat{y}_{1,N};\widehat{z}_N)
    \end{equation}
    \begin{equation}
    G_k^N(z,y_2) - G_k^N(z,y_1) - h^N(x_2,y_2) + h^N(x_2,y_1) \leq - \log (1-B_k^N(\overline{x}_{2,N}, \widehat{y}_{1,N};\widehat{z}_N-1))
    \end{equation}
    \begin{equation}
    G_k^N(z,y_2) - G_k^N(z,y_1) - h^N(x_2,y_2) + h^N(x_2,y_1) \geq \log (1-A_k^N(\overline{x}_{2,N},\widehat{y}_{2,N};\widehat{z}_N+1)).
    \end{equation}
for $x_1, x_2, y_1, y_2, z \in \R$ such that $0< x_1 \leq x_2$ and $y_2 \geq y_1 \geq z$.
\end{cor}

\subsection{Change of coordinates}
As the log-gamma line ensemble converges to the Airy line ensemble, our ultimate goal is to show that the free energy of the log-gamma line ensemble converges to the last passage value of the Airy line ensemble. Since the mesh size of the discrete model is shrinking in the line ensemble convergence, we need to establish change-of-coordinates formulas for the free energy.

\begin{define}
    Given a discrete subset $\D \subset \R$ and let $x, y \in \mathbb{D}$ such that $x \leq y$. Let $\ell \geq m$ be two positive integers. We denote by $\Pi_{\D}[(x,\ell) \rightarrow (y,m)]$ the set of up-right paths from $(x,\ell)$ to $(y,m)$ on the lattice $\D \times \lb 1,n \rb$.
\end{define}

There is an injective map from $\Pi_{\D}[(x,\ell) \rightarrow (y,m)]$ to $\D^{\ell -m}$, $\pi \rightarrow (t_{\ell}, \cdots, t_{m+1})$ where $t_i = \max \{t \in \lb x,y \rb| (t,i) \in \pi \}$. For notational convenience, we also set $t_{\ell+1} =x, t_m = y.$

\begin{define}
    For any discrete subset $\D \subset \R$ and $n$ functions $f_1, \cdots, f_n: \D  \rightarrow \R$ and $\pi \in \Pi_{\D}[(x,\ell) \rightarrow (y,m)]$, define
    \begin{equation}
    f(\pi) = \sum_{j=m}^{\ell} f_j(t_j) - f_j(t_{j+1}')
    \end{equation}
    where $t_{j+1}' = \max \{ t \in \D| t < t_{j+1}\}$ and $f_i(x) = 0$ for all $i \in \lb 1, n \rb$ and $x \not \in \D$.
    
    We define the free energy from $(x,\ell)$ to $(y,m)$ as
    \begin{equation}
    f[(x,\ell) \rightarrow (y,m)] := \log \sum_{\pi \in \Pi_{\D}[(x,\ell) \rightarrow (y,m)]} \exp( f(\pi)).
    \end{equation}
\end{define}

Similarly, we can define the reverse free energy on any discrete lattice.
\begin{define}
    For any discrete subset $\D \subset \R$ and $n$ functions $f_1, \cdots, f_n: \D \rightarrow \R$, for $x, y \in \D$ where $x \leq y$ and $m \in \lb 1, n \rb$, we define the reverse free energy from $(x, 1)$ to $(y,m)$:
    \begin{equation}
    f[(x,1) {\searrow} (y,m)] := - \log \sum_{x= t_0 < t_1 < \cdots < t_{m-1} \leq y; t_i \in \D} \exp([\sum_{i=1}^{m} f_{i}(t_{i-1}) - f_i(t_i')])
    \end{equation}
    where $t_{i}' = \max \{ t \in \D| t < t_{i}\}$ for $i \in \lb 1, k \rb$ and $t_{m}' = y$.
\end{define}

\begin{rmk}
    Observe that the definitions of the free energy and the reverse free energy on the discrete lattice generalize those on the integer lattice. 
    We also introduce the $\beta$-free energy and $\beta$-reverse free energy by
    \begin{equation}
        f[(x,\ell) \xrightarrow{\beta} (y,m)] := \beta^{-1} (\beta f)[(x,\ell) \rightarrow (y,m)],
    \end{equation}
    \begin{equation}
        f[(x,1) \overset{\beta}{\searrow} (y,m)] := \beta^{-1} (\beta f)[(x,1) \searrow (y,m)],
    \end{equation}
    where $\beta f$ denotes the function $f$ multiplied by the scalar $\beta$.
\end{rmk}

\begin{lemma}\label{rescale1}
    Fix $n$ functions $f_1, \cdots, f_n: \Z_{\geq 1} \rightarrow \R$ and constants $a_1, a_2 > 0$, $a_3, a_4 \in \R$ and $\{a_{5,i}\}_{i \in \N}$. Define the functions $g = (g_1, \cdots, g_n)$ by
    \begin{equation}
    g_i(x) = a_1f_i(a_2x + a_3) + a_4x + a_{5,i}
    \end{equation}
    for $x \in \D := \{a_2^{-1}(z-a_3): z \in \Z_{\geq 1} \}$. Then for all $x,y \in \D$ and $\ell, k \in \lb 1 ,n \rb$ with $x \leq y$ and $\ell \geq k$, we have
    \begin{align}
g[(x,\ell) \xrightarrow{\beta} (y,k)]
  &= a_1 f[(a_2x+a_3,\ell) \xrightarrow{a_1\beta} (a_2y+a_3,k)]
     + a_4(y-x) + a_4 a_2^{-1}(\ell - k + 1), \label{eq_330}\\
g[(x,1) \bd (y,k)]
  &= a_1 f[(a_2x+a_3,1)\overset{a_1\beta}{\searrow} (a_2y+a_3,k)]
     + a_4(y-x) - a_4 a_2^{-1}(k-1). \label{eq_331}
\end{align}
\end{lemma}
\begin{proof}
We will prove equation (\ref{eq_330}) and the proof for equation (\ref{eq_331}) follows analogously.
    \[
    \begin{split}
        g[(x,\ell) \xrightarrow{\beta} (y,k)]
        &=  \beta^{-1} \log \sum_{x \leq t_{\ell} \leq \cdots \leq t_{k+1} \leq y; t_i \in C} \exp\bigg(\beta \sum_{j=k}^{\ell} g_j(t_j) - g_j(t_{j+1}-a_2^{-1})\bigg)\\
        &=  \beta^{-1} \log \sum_{ \substack{a_2x +a_3 \leq t_{\ell} \leq \cdots \leq t_{k+1} \leq a_2y+a_3 \\ t_i \in \Z_{\geq 1}} } \exp\bigg(\beta \sum_{j=k}^{\ell} g_j(\frac{t_j-a_3}{a_2}) - g_j(\frac{t_{j+1} - a_3-1}{a_2})\bigg)\\
        &=  \beta^{-1} \log \sum_{a_2x +a_3 \leq t_{\ell} \leq \cdots \leq t_{k+1} \leq a_2y+a_3; t_i \in \Z_{\geq 1}} \exp\bigg(a_1\beta[\sum_{j=k}^{\ell} f_j(t_j)  - f_j(t_{j+1}-1)]\bigg)\\
        & \quad \times \exp\bigg(\beta a_4a_2^{-1}\sum_{j=k}^{\ell} (t_j-t_{j+1}+1)\bigg)\\
        &=  a_1 f[(a_2x+a_3,\ell)\xrightarrow{a_1\beta} (a_2y+a_3,k))] + a_4(y-x) + a_4a_2^{-1}(\ell - k+1)\\
    \end{split}
    \]
\end{proof}

\subsection{Scaling limit of free energy}
Recall that we have defined the log-gamma line ensemble $\A^N = (\A_1^N, \A_2^N, \cdots)$ as follows. For $i \in \lb 1, N \rb$, 
\begin{equation}
\A^N_i(s)=
\begin{cases}
\sigma_pN^{-1/3}(L_i^N(-\frac{1}{2}N)+(p/2)N) & \text{if } s \leq -\frac{1}{2}N^{1/3}\\
\sigma_pN^{-1/3}(L_i^N(sN^{2/3})-psN^{2/3}) & \text{if } s \in [-\frac{1}{2}N^{1/3},\frac{1}{2}N^{1/3}]\\
\sigma_pN^{-1/3}(L_i^N(\frac{1}{2}N)-(p/2)N) & \text{if } s \geq \frac{1}{2}N^{1/3}.
\end{cases}
\end{equation}
For $i \geq N+1$, we define $\A^N_i \equiv 0$.  

We can also consider the projection of the line ensemble $\A^N = (\A_1^N, \A_2^N, \cdots)$, a random variable taking values in $C(\N \times \R)$, onto the discrete space $C(\N \times \D)$ where
$\D = \{s \in \R: sN^{2/3} \in \Z\}$. This yields the projected discrete line ensemble $\AAA^N = (\AAA_1^N, \AAA_2^N, \cdots)$ where for $i \in \lb 1, N \rb$
\begin{equation}
\begin{split}
\AAA^N_i(s)=
\begin{cases}
\sigma_pN^{-1/3}(Wf_i^N(2N - \lfloor \frac{N}{2} \rfloor)+2Nh_\theta(1)+(p/2)N) & \text{if } s \in (-\infty,-\frac{1}{2}N^{1/3}] \cap \D\\
\sigma_pN^{-1/3}(Wf_i^N(sN^{2/3}+2N)+2Nh_\theta(1)-psN^{2/3}) & \text{if } s \in [-\frac{1}{2}N^{1/3},\frac{1}{2}N^{1/3}] \cap \D\\
\sigma_pN^{-1/3}(Wf_i^N(2N + \lfloor \frac{N}{2} \rfloor)+2Nh_\theta(1)-(p/2)N) & \text{if } s \in [\frac{1}{2}N^{1/3}, \infty) \cap \D.
\end{cases}
\end{split}
\end{equation}
For $i \geq N+1$, $\AAA^N_i \equiv 0$.  

In the following calculation, we will always choose $N$ large enough so that $s \in [-\frac{1}{2}N^{1/3}, \frac{1}{2}N^{1/3}]$.

\begin{lemma}\label{D to C}
Let $x,y,z \in \R$ such that $x>0, y \geq z$. Recall $\widehat{z}_N = \lfloor N^{2/3}zq^{-2} \rfloor + 2N$, $\widehat{y}_N = \lfloor N^{2/3}yq^{-2} \rfloor + 2N$ and $\overline{x}_N = \lfloor N^{2/3}xq^{-2} \rfloor+ 1$. On the event that $\A^N$ converges to $\A^q$ uniformly over compact sets, we have
    \begin{equation}
    \lim_{N \rightarrow \infty} \A^N[(\frac{\widehat{z}_N - \overline{x}_N +1}{N^{2/3}}-2N^{1/3},1) \overset{\sigma_p^{-1}N^{1/3}}{\searrow} (\frac{\widehat{z}_N}{N^{2/3}}-2N^{1/3},k+1)] = \A^q[(zq^{-2}-xq^{-2},1) \rightarrow_f (zq^{-2},k+1)]
    \end{equation}
    \begin{equation}
    \lim_{N \rightarrow \infty} \A^N[(\frac{\widehat{z}_N}{N^{2/3}}-2N^{1/3},k) \overset{\sigma_p^{-1}N^{1/3}}{\rightarrow} (\frac{\widehat{y}_N}{N^{2/3}}-2N^{1/3},1)] = \A^q[(zq^{-2},k) \xrightarrow{\infty} (yq^{-2},1)]
    \end{equation}
    where the backwards first passage time and the last passage time on a continuous line ensemble are defined in Definition \ref{def_fpp} and Definition \ref{def_lpp}. Together with Theorem \ref{thm_airy_line_ensemble}, this implies the convergence in distribution.
\end{lemma}
We begin by defining the $\beta$-free energy and the reverse $\beta$-free energy on the continuous line ensemble.

\begin{define}
    Let $x, y \in \R$ such that $x < y$. Let $\ell \geq m$ be two positive integers. We denote by $\PPi[(x,\ell) \rightarrow (y,m)]$ the set of up-right paths $\pi:[x,y] \rightarrow \lb m,\ell \rb$ where $\pi$ is a decreasing c\`adl\`ag function with $\pi(x)\leq \ell$ and $\pi(y) = m$. We denote by $\PPi[(x,1) \searrow (y,m)]$ the set of down-right paths $\sigma:[x,y] \rightarrow \lb 1,m \rb$ where $\sigma$ is an increasing c\`adl\`ag function with $\sigma(x) \geq 1$ and $\sigma(y) = m$. 
\end{define}

There is an injective map from $\PPi[(x,\ell) \rightarrow (y,m)]$ to $\R^{\ell -m}$, $\pi \rightarrow (t_{\ell}, \cdots, t_{m+1})$ where $t_i = \inf \{t \in [ x,y ]| \pi(t) \leq i - 1 \}$ for all $i \in \lb m+1, \ell \rb$. For convenience, we set $t_{\ell+1} =x, t_m = y.$ There is also an injective map from $\PPi[(x,1) \searrow (y,m)]$ to $\R^{m-1}$, $\sigma \rightarrow (\tau_1, \cdots, \tau_{m-1})$ where $\tau_i = \inf \{\tau \in [ x,y ]| \sigma(\tau) \geq i - 1 \}$ for all $i \in \lb 1, m-1 \rb$. We also set $
\tau_0 =x, \tau_m = y.$

\begin{define}
    Let $f = (f_1, f_2, \cdots, f_n)$ where $f_i: \R \rightarrow \R$ is continuous for all $i \in \lb 1, n \rb$ and $\ell, m \in \lb 1 , n \rb$ such that $\ell \geq m$. For $\pi \in \PPi[(x,\ell) \rightarrow (y,m)]$ and $\sigma \in \PPi[(x,1) \searrow (y,m)]$, we define
    \begin{equation}
    f(\pi) := \sum_{j=m}^{\ell} f_j(t_j) - f_j(t_{j+1}), \quad f(\sigma) := \sum_{j=1}^m f_j(\tau_j) - f_j(\tau_{j-1}).
    \end{equation}
We define the free energy from $(x,\ell)$ to $(y,m)$ and the reverse free energy from $(x,1)$ to $(y,m)$
    \begin{align}
    &f[(x,\ell) \rightarrow (y,m)] := \log \int_{\PPi[(x,\ell) \rightarrow (y,m)]} \exp(f(\pi)) d\pi, \\ &f[(x,\ell) {\searrow} (y,m)] :=- \log \int_{\PPi[(x,1) \searrow (y,m)]} \exp(- f(\sigma)) d\sigma,
    \end{align}
    where $d\pi$ and $d\sigma$ are Lebesgue measures on the subsets of $\R^{\ell - m}$ and $\R^{m-1}$.

    The $\beta$-free energy and $\beta$-reverse free energy are defined analogously to the discrete case:
    \begin{equation}
        f[(x,\ell) \xrightarrow{\beta} (y,m)] := \beta^{-1} (\beta f)[(x,\ell) \rightarrow (y,m)],
    \end{equation}
    \begin{equation}
        f[(x,1) \overset{\beta}{\searrow} (y,m)] := \beta^{-1} (\beta f)[(x,1) \searrow (y,m)],
    \end{equation}
    where $\beta f$ denotes the function $f$ multiplied by the scalar $\beta$.
\end{define}

Now we are ready to prove Lemma \ref{D to C}.
\begin{proof}
    By Laplace method and continuity of the $\A^q$, we have 
    \begin{equation}
        \lim_{N \rightarrow \infty} \A^q[(\frac{\widehat{z}_N - \overline{x}_N +1}{N^{2/3}}-2N^{1/3},1) \overset{\sigma_p^{-1}N^{1/3}}{\searrow} (\frac{\widehat{z}_N}{N^{2/3}}-2N^{1/3},k+1)] = \A^q[(zq^{-2}-xq^{-2},1) \rightarrow_f (zq^{-2},k+1)].
    \end{equation}

Let $\D =\{sN^{-2/3}q^{2}: s\in \Z\}$, $a = \frac{\widehat{z}_N - \overline{x}_N +1}{N^{2/3}}-2N^{1/3}$, $b = \frac{\widehat{z}_N - \overline{x}_N +1}{N^{2/3}}-2N^{1/3}$, $\Pi = \Pi_{\D}[(a,1) \searrow (b,k+1)]$ and $\PPi = \PPi[(a,1) {\searrow} (b,k+1)]$. Then, 
    \begin{equation}
    \begin{split}
    &\AAA^N[(a,1) \overset{\sigma_p^{-1}N^{1/3}}{\searrow} (b,k+1)]  -  \A^q[(a,1) \overset{\sigma_p^{-1}N^{1/3}}{\searrow} (b,k+1)]  \\
    & \quad \quad \quad = - \sigma_pN^{-1/3} \bigg(\log \frac{ \sum_{\pi \in \Pi} \exp(-\AAA^N(\pi) \sigma_p^{-1} N^{1/3})}{ \sum_{\pi \in \Pi} \exp(-\A^q(\pi) \sigma_p^{-1} N^{1/3})} + \log \frac{ \sum_{\pi \in \Pi} \exp(-\A^q(\pi) \sigma_p^{-1} N^{1/3})}{ \int_{\PPi} \exp(-\A^q(\tilde{\pi}) \sigma_p^{-1} N^{1/3})d\tilde{\pi}}\bigg).
    \end{split}
    \end{equation}
    Because $\A^N$ converges to $\A^q$ uniformly on compact sets, for any $\epsilon > 0$, there exists a $N_0 \in \N$ such that for $N \geq N_0$ and for $\pi \in \Pi$,$
    |\AAA^N(\pi) - \A^q(\pi)| \leq 4k \epsilon.$
    Thus,
    \begin{equation}
   e^{-4k\epsilon \sigma_p^{-1} N^{1/3}} \leq \frac{ \sum_{\pi \in \Pi} \exp(-\AAA^N(\pi) \sigma_p^{-1} N^{1/3})}{ \sum_{\pi \in \Pi} \exp(-\A^q(\pi) \sigma_p^{-1} N^{1/3})} \leq e^{4k\epsilon \sigma_p^{-1} N^{1/3}}
    \end{equation}
    \begin{equation}
    \bigg|-\sigma_pN^{-1/3}\log \frac{ \sum_{\pi \in \Pi} \exp(-\AAA^N(\pi) \sigma_p^{-1} N^{1/3})}{ \sum_{\pi \in \Pi} \exp(-\A^q(\pi) \sigma_p^{-1} N^{1/3})}\bigg| \leq 4k \epsilon.
    \end{equation}
    On the other hand, because $\A^q$ is continuous and thus uniformly continuous on compact sets, there exists $N_0 \in \N$ such that for $u,v \in [zq^{-2}-xq^{-2}-1, zq^{-2}+1]$ such that $|u-v| \leq 2kN_0^{-2/3}$ and $i \in \lb 1 , k+1 \rb$, 
    \begin{equation}
    |\A^q_i(u) - \A^q_i(v)| \leq \epsilon
    \end{equation}
    For $N \geq N_0$, when $\tilde{\pi} \in \mathscr{P}[(a,1)\searrow (b, k+1)]$ and $\pi \in \Pi_{\D}[(a,1)\searrow (b, k+1)]$ are sufficiently closed in $\R^{\ell-m}$, we have 
    \begin{equation}
    |\A^q(\pi) - \A^q(\tilde{\pi})| \leq 4k\epsilon.
    \end{equation}
    Notice that
    \begin{equation}
    \sum_{\pi \in \Pi} \exp((-\A^q(\pi) - 4k \epsilon) \sigma_p^{-1} N^{1/3}) \leq \int_{\PPi} \exp(-\A^q(\tilde{\pi}) \sigma_p^{-1} N^{1/3})d\tilde{\pi} 
    \end{equation}
    and
    \begin{equation}
    \int_{\PPi} \exp(-\A^q(\tilde{\pi}) \sigma_p^{-1} N^{1/3})d\tilde{\pi} \leq \frac{k!}{N^{2k/3}}\sum_{\pi \in \Pi} \exp((-\A^q(\pi) + 4k \epsilon) \sigma_p^{-1} N^{1/3}).
    \end{equation}
    Thus,
    \begin{equation}
    -4k \epsilon \leq  -  \sigma_pN^{-1/3} \log \frac{ \sum_{\pi \in \Pi} \exp(-\A^q(\pi) \sigma_p^{-1} N^{1/3})}{ \int_{\PPi} \exp(-\A^q(\tilde{\pi}) \sigma_p^{-1} N^{1/3})d\tilde{\pi}} \leq  4k \epsilon + \sigma_pN^{-1/3}\log(k!N^{-2k/3}). 
    \end{equation}
We can take $N$ large enough so that $\sigma_pN^{-1/3} \log (k!N^{-2k/3})$ are absorbed into the $\epsilon$ term. Thus, we have
    \begin{align}
          &\lim_{N \rightarrow \infty}\AAA^N[(\frac{\widehat{z}_N - \overline{x}_N +1}{N^{2/3}}-2N^{1/3},1) \overset{\sigma_p^{-1}N^{1/3}}{\searrow} (\frac{\widehat{z}_N}{N^{2/3}}-2N^{1/3},k+1)] \\
        = &\lim_{N \rightarrow \infty} \A^q[(\frac{\widehat{z}_N - \overline{x}_N +1}{N^{2/3}}-2N^{1/3},1) \overset{\sigma_p^{-1}N^{1/3}}{\searrow} (\frac{\widehat{z}_N}{N^{2/3}}-2N^{1/3},k+1)]\\
        = &\A^q[(zq^{-2}-xq^{-2},1) \rightarrow_f (zq^{-2},k+1)].
    \end{align}
    The proof for the second equation is analogous.
\end{proof}

\begin{lemma}\label{lemma_F_to_Airy}
    For any $k \in \Z_{\geq 1}$, $F_k^N(x,z)$ converges in distribution to $\A[(0,k+1) \rightarrow (x,1)] +2zx$ as $N \rightarrow \infty$.
\end{lemma}

\begin{proof}
    By Lemma \ref{key1}, 
    \begin{equation}
    \begin{split}
    &F_k^N(x,z)= -2^{-1/2}q\sigma_pN^{-1/3} \left[W(R_{\widehat{z}_N}f^N)[(\widehat{z}_N - \overline{x}_N + 1,1) \searrow (\widehat{z}_N,k+1)]-p(\overline{x}_N - N^{2/3}k)\right]
    \end{split}
    \end{equation}
    Because $d_{i,j}$ is identically and independently distributed for all $i, j$, we have that $d \dq R_{\widehat{z}_N}d$. This implies $f \dq R_{\widehat{z}_N}f$. Thus, 
    \begin{equation}
    W(R_{\widehat{z}_N}f^N)[(\widehat{z}_N - \overline{x}_N + 1,1) \searrow (\widehat{z}_N,k+1)] \dq  Wf^N[(\widehat{z}_N - \overline{x}_N + 1,1) \searrow (\widehat{z}_N,k+1)]
    \end{equation}
    Since for $x \in \Z_{\geq 1}$, we have
    \begin{equation}
    Wf_i^N(x) = \sigma_p^{-1} N^{1/3} \A^N_i(\frac{x}{N^{2/3}} -2 N^{1/3}) -2N(h_{\theta}(1)+p) + px,
    \end{equation}
    it follows from Lemma \ref{rescale1} that
    \begin{equation}
    \begin{split}
        &Wf^N[(\widehat{z}_N - \overline{x}_N + 1,1) \searrow (\widehat{z}_N,k+1)] \\
        &= \sigma_p^{-1} N^{1/3}\A^N[(\frac{\widehat{z}_N - \overline{x}_N +1}{N^{2/3}}-2N^{1/3},1) \overset{\sigma_p^{-1}N^{1/3}}{\searrow} (\frac{\widehat{z}_N}{N^{2/3}}-2N^{1/3},k+1)]+ p(\overline{x}_N-1) - pN^{2/3}k.
    \end{split}
    \end{equation}
    Rearranging yields: 
    \begin{equation}
    \begin{split}
        &-\sigma_pN^{-1/3}[Wf^N[(\widehat{z}_N - \overline{x}_N + 1,1) \searrow  (\widehat{z}_N,k+1)] - p(\overline{x}_N-1 - N^{2/3}k)]\\
        &= -\A^N[(\frac{\widehat{z}_N - \overline{x}_N +1}{N^{2/3}}-2N^{1/3},1) \overset{\sigma_p^{-1}N^{1/3}}{\searrow} (\frac{\widehat{z}_N}{N^{2/3}}-2N^{1/3},k+1)].
    \end{split}
    \end{equation}
    By Lemma \ref{D to C},
    \begin{equation}
    \begin{split}
        -\sigma_pN^{-1/3}[Wf^N[(\widehat{z}_N - \overline{x}_N + 1,1) \searrow  (\widehat{z}_N,k+1)] - p(\overline{x}_N - N^{2/3}k)]
    \end{split}
    \end{equation}
    converges in distribution to $-\A^q[(zq^{-2}-xq^{-2},1) \rightarrow_f (zq^{-2},k+1)]$.
    
    Recall that $\A^q_i(x) = 2^{1/2}q^{-1}\A_i(q^2x)$. Thus,
    \begin{equation}
    \A^q[(zq^{-2}-xq^{-2},1) \rightarrow_f (zq^{-2},k+1)] = 2^{1/2}q^{-1}\A[(z-x,1) \rightarrow_f (z, k+1)]
    \end{equation}
    By the flip symmetry of Airy line ensemble,
    \begin{equation}
    -\A[(z-x,1)\rightarrow_f (z,k+1)] \dq \A[(-z,k+1) \rightarrow (x-z,1)].
    \end{equation}
    Furthermore, since the process $t \mapsto A(t)+t^2$ is stationary, it follows that
    \begin{equation}
    \A[(-z,k+1) \rightarrow (x-z,1)] \dq \A[(0,k+1) \rightarrow (x,1)] + 2zx.
    \end{equation}
    Thus,
    \begin{equation}
    -2^{-1/2}q\sigma_pN^{-1/3}[Wf^N[(\widehat{z}_N - \overline{x}_N + 1,1) \searrow  (\widehat{z}_N,k+1)] - p(\overline{x}_N - N^{2/3}k)]
    \end{equation}
    converges to $\A[(0,k+1) \rightarrow (x,1)] +2zx$ in distribution.
\end{proof}

\begin{lemma}
    On the event that $\A^N$ converges to $\A^q$ uniformly over compact subsets, for any $z, y, y' \in \R$ with $y',y \geq z$, we have
    \begin{equation}
    G^N_k(z,y') - G^N_k(z,y) 
    \end{equation}
    converges to
    \begin{equation}
    \A[(z,k) \rightarrow (y',1)] - \A[(z,k) \rightarrow (y,1)].
    \end{equation}
\end{lemma}

\begin{proof}
    By definition, 
    \begin{equation}
    \begin{split}
        G_k^N(z,y') - G^N_k(z,y)
        &=2^{-1/2}q\sigma_pN^{-1/3}\big[Wf^N[(\widehat{z}_N, k) \rightarrow (\widehat{y'}_N,1)] -p(\widehat{y'}_N - \widehat{z}_N + N^{2/3}k)\big] \\
         & \quad \quad 
        -2^{-1/2}q\sigma_pN^{-1/3}\big[Wf^N[(\widehat{z}_N, k) \rightarrow (\widehat{y}_N,1)] -p(\widehat{y}_N - \widehat{z}_N + N^{2/3}k)\big].
    \end{split}
    \end{equation}
    Recall that for $x \in \Z_{\geq 1}$,
    \begin{equation}
    Wf_i^N(x) = \sigma_p^{-1} N^{1/3} \A^N_i(\frac{x}{N^{2/3}} -2 N^{1/3}) -2N(h_{\theta}(1)+p) + px.
    \end{equation}
    Apply Lemma \ref{rescale1}, we obtain
    \[
    \begin{split}
        \A^N[(\frac{\widehat{z}_N}{N^{2/3}} - 2N^{1/3},k) \xrightarrow{\sigma_p^{-1}N^{1/3}} (\frac{\widehat{y'}_N}{N^{2/3}} - 2N^{1/3},1)]= \sigma_pN^{-1/3} \big[Wf^N[(\widehat{z}_N, k) \rightarrow (\widehat{y'}_N,1)] -p(\widehat{y'}_N - \widehat{z}_N + N^{2/3}k)\big]
    \end{split}
    \]
    \[
    \begin{split}
        \A^N[(\frac{\widehat{z}_N}{N^{2/3}} - 2N^{1/3},k) \xrightarrow{\sigma_p^{-1}N^{1/3}} (\frac{\widehat{y}_N}{N^{2/3}} - 2N^{1/3},1)]&= \sigma_pN^{-1/3} \big[Wf^N[(\widehat{z}_N, k) \rightarrow (\widehat{y}_N,1)] -p(\widehat{y}_N - \widehat{z}_N + N^{2/3}k)\big].
    \end{split}
    \]
    By Lemma \ref{D to C}, we know that as $N \rightarrow \infty$, $G_k^N(z,y') - G_k^N(z,y)$ converges to 
    \begin{equation}
    2^{-1/2}q\A^q[(q^{-2}z,k) \rightarrow (q^{-2}y',1)] - 2^{-1/2}q\A^q[(q^{-2}z,k) \rightarrow (q^{-2}y,1)] 
    \end{equation}
    which is equal to
    \begin{equation}
    \A[(z,k) \rightarrow (y',1)] - \A[(z,k) \rightarrow (y,1)].
    \end{equation}
\end{proof}

\section{Convergence to the Airy sheet}
\subsection{Tightness}

Because $h^N(x,y)$ is only a piecewise constant function, we will work with its linear interpolation $\tilde{h}^N(x,y)$ such that $\tilde{h}^N(x,y)$ agree with $h^N(x,y)$ on $(x,y)$ whenever $N^{2/3}xq^{-2}, N^{2/3}yq^{-2} \in \Z$. Recall that $C(\R^2, \R)$ denotes the set of continuous function from $\R^2$ to $\R$, equipped with the topology of uniform convergence over compact subsets. 

\begin{prop}
    $\tilde{h}^N(x,y)$ is tight in $C(\R^2, \R)$.
\end{prop}

\begin{proof}
    It suffices to show that for all $L > 0$, $\tilde{h}^N(x,y)$ restricted on $Q = [-L,L]^2$ is tight in $C(Q, \R)$. Building on that, we need to show one-point tightness of $\tilde{h}^N$ and establish the modulus of continuity control. Since
    \begin{equation}
    \begin{split}
    \tilde{h}^N(0,0) 
= 2^{-1/2}q\sigma_pN^{-1/3} \left[\log Z[(1, 1) \rightarrow (2N,2N)] + h_\theta(1)\left(2N-\frac{1}{2}\right) \right],
    \end{split}
    \end{equation}
    the tightness of $\tilde{h}^N(0,0)$ follows directly from its convergence in \cite[Theorem 1.2]{barraquand2021fluctuations}.

    Let $g^N(y) = \tilde{h}^N(0,y)$. We know that for $N^{2/3}yq^{-2} \in \Z$, 
    \begin{equation}
    \begin{split}
        g^N(y) = 2^{-1/2}q\sigma_pN^{-1/3} \left[\log Z[(1,1) \rightarrow (2N + N^{2/3}yq^{-2},2N)]+ 2N h_\theta(1)  +h_\theta'(1)(N^{2/3}yq^{-2}-1)\right].
    \end{split}
    \end{equation}
    By \cite[Theorem 1.10]{barraquand2023spatial}, $g^N(y)$ is tight in $C([-2L,2L],\R)$. Thus, for each $\epsilon > 0$ and $\eta \in (0,1)$, there exists a $\delta > 0$ and $N' \in \N$ such that for $N \geq N'$ we have
    \begin{equation}
    \PP(w(g^N, \delta) \geq \epsilon) \leq \eta,
    \end{equation}
    where the modulus of continuity is defined by
    $
    w(g,\delta) = \sup_{x,y \in [-2L,2L], |x-y| \leq \delta} |g(x) - g(y)|.
    $

    We now show that the modulus of continuity of $\tilde{h}^N$ can be controlled via the modulus of continuity of $g^N$. It suffices to develop modulus of continuity control for pair of points $(x,y)$ such that $N^{2/3}xq^{-2},N^{2/3}yq^{-2} \in \Z$ because $\tilde{h}^N$ is the linear interpolation of these pairs.

     Let $L_+$ be the largest element in the set $[-L,L] \cap N^{-2/3}q^2\Z$ and $L_-$ be the smallest element in the set $[-L,L] \cap N^{-2/3}q^2\Z$. Choose $N$ large enough such that $\lfloor N^{2/3}L_+ \rfloor + 1 \leq \lfloor N^{2/3}L_- \rfloor +2N$. By Corollary \ref{quadrangle}, for any $x_1,x_2,y_1,y_2 \in [-L,L] \cap N^{-2/3}q^2\Z$ where $x_1 \leq x_2$ and $y_1 \leq y_2$, we have
    \begin{align}
    h^N(L_+,y_1) - h^N(L_+,y_2) &\leq h^N(x,y_1) - h^N(x,y_2) \leq h^N(L_-,y_1) - h^N(L_-,y_2)\\
    h^N(x_1,L_+) - h^N(x_2,L_+) &\leq h^N(x_1,y) - h^N(x_2,y) \leq h^N(x_1,L_-) - h^N(x_2,L_-).
    \end{align}
Now for $x_1,x_2,y_1,y_2 \in [-L,L] \cap N^{-2/3}q^2\Z$,
    \begin{equation}
    \begin{split}
        |h^N(x_1,y_1) - h^N(x_2,y_2)|
        &\leq|h^N(x_1,y_1) - h^N(x_1,y_2)| + |h^N(x_1,y_2) - h^N(x_2,y_2)|\\
        &\leq \max_{z \in \{L_-,L_+\}}|h^N(z,y_1) - h^N(z,y_2)| + \max_{z \in \{L_-,L_+\}}|h^N(x_1,z) - h^N(x_2,z)|.
    \end{split}
    \end{equation}
Thus,
    \begin{equation}
    \begin{split}
        &\PP\bigg(\sup_{\substack{(x_1,y_1),(x_2,y_2) \in [-L,L]^2 \\ |(x_1,y_1) - (x_2,y_2)| \leq \delta}} |h^N(x_1,y_1) - h^N(x_2,y_2)| \geq \epsilon \bigg)\\
        &\leq \PP \bigg(\sup_{\substack{(x_1,y_1),(x_2,y_2) \in [-L,L]^2 \\ |x_1 - x_2| \leq \delta, |y_1 - y_2| \leq \delta}} \max_{z \in \{-L,L\}}|h^N(z,y_1) - h^N(z,y_2)| + \max_{z \in \{-L,L\}}|h^N(x_1,z) - h^N(x_2,z)| \geq \epsilon \bigg)\\
        & \leq \PP \bigg(\sup_{ \substack{y_1,y_2 \in [-L,L] \\ |y_1 - y_2| \leq \delta}} |h^N(L_+,y_1) - h^N(L_+,y_2)| \geq \epsilon \bigg) +\PP \bigg(\sup_{ \substack{y_1,y_2 \in [-L,L] \\ |y_1 - y_2| \leq \delta}} |h^N(L_-,y_1) - h^N(L_-,y_2)| \geq \epsilon \bigg) \\
        & \quad  + \PP \bigg(\sup_{ \substack{x_1,x_2 \in [-L,L] \\ |x_1 - x_2| \leq \delta}} |h^N(x_1,L_+) - h^N(x_2,L_+)| \geq \epsilon \bigg) +\PP \bigg(\sup_{ \substack{x_1,x_2 \in [-L,L] \\ |x_1 - x_2| \leq \delta}} |h^N(x_1,L_-) - h^N(x_2,L_-)| \geq \epsilon \bigg) \\
        &\leq 4 \eta.
    \end{split}
    \end{equation}
    The final inequality follows from the modulus of continuity control for $g^N$, together with the distributional symmetries
    \begin{equation}
    h^N(L_+,\cdot) \dq h^N(0,\cdot-L_+)\end{equation}
\begin{equation}
h^N(\cdot, L_+) \dq h^N(0,L_+-\cdot)\end{equation}
viewed as processes on discrete lattice $N^{-2/3}q^{2}\Z \times N^{-2/3}q^{2}\Z $.
\end{proof}

\subsection{Convergence}
Let $R_k^N (x,z) = F_k^N(x,z) - 2zx - 2^{3/2}k^{1/2}x^{1/2}.$ We obtain the following inequalities as consequences of Corollary \ref{h_S}.

\begin{cor}\label{hR}
Let $x_1, x_2>0$, $y \in \R$ and $z = - 2^{-1/2}k^{1/2}x_1^{-1/2}$. Then 
    \begin{equation}
    F_k^N(x_2,z) - F_k^N(x_1,z) =R_k^N(x_2,z) - R_k^N(x_1, z) -2^{1/2}k^{1/2}x_1^{1/2}(1-x_2^{1/2}x_1^{-1/2})^2.
    \end{equation}
    If $x_1 \geq x_2$, then
    \begin{equation}
    \log A_k^N (\overline{x}_{2,N},\widehat{y};\widehat{z}_N) \leq h^N(x_1,y) - h^N(x_2, y) +R_k^N(x_2,z) - R_k^N(x_1, z) - 2^{1/2}k^{1/2}x_1^{1/2}(1-x_2^{1/2}x_1^{-1/2})^2.
    \end{equation}
    If $x_1 \leq x_2$, then
    \begin{equation}
    \log B_k^N (\overline{x}_{2,N},\widehat{y};\widehat{z}_N) \leq h^N(x_1,y) - h^N(x_2,y) + R_k^N(x_2,z) - R_k^N(x_1,z) - 2^{1/2}k^{1/2}x_1^{1/2}(1-x_2^{1/2}x_1^{-1/2})^2.
    \end{equation}
\end{cor}

\begin{prop}\label{coupling}
    Fix a sequence $\{N_i\}_{i \in \Z_{\geq 1}}$. There exists a subsequence $\{M_i\}_{i \in \Z_{\geq 1}} \subset \{N_i\}_{i \in \Z_{\geq 1}}$ and a coupling of $\{\tilde{h}^{M_i}, \A^{M_i}, R_k^{M_i}\}$ and the Airy line ensemble $\A$ such that the following statements hold almost surely:
    \begin{enumerate}
        \item $\A^M$ converges to $\A$ in $C(\N \times \R, \R)$.
        \item $\tilde{h}^M$ converges to some limit $h$ in $C(\R^2, \R)$ and $h(0, \cdot) = \A_1(\cdot).$ 
        \item For all $k \in \Z_{\geq 1}$ and $x_1, x_2 \in \Q^{+}$, the scaled component function $R_k^M(x_1,-2^{-1/2}k^{1/2}x_2^{-1/2})$ converges to some limit $R_k(x_1,-2^{-1/2}k^{1/2}x_2^{-1/2})$ and
    \begin{equation}\label{Borel_Cantelli}
    \lim_{k \rightarrow \infty} |k^{-1/2}R_k(x_1, -2^{-1/2}k^{1/2}x_2^{-1/2})| =  0.
    \end{equation} 
    \end{enumerate}
\end{prop}

\begin{proof}
By Lemma \ref{lemma_F_to_Airy}, we know that for fixed $x_1, x_2 > 0$, $R_k^N(x_1,-2^{-1/2}k^{1/2}x_2^{-1/2})$ converges in distribution as $N \rightarrow \infty$. Together with the tightness of $\{\tilde{h}^N\}_{N \in \Z_{\geq 1}}$ and the convergence of Airy line ensemble $\A^N$ in Theorem \ref{thm_airy_line_ensemble}, we can apply Skorokhod's representation theorem to extract a subsequence $\{M_i\}_{i \in \Z_{\geq 1}} \subset \{N_i\}_{i \in \Z_{\geq 1}}$ and construct a coupling under which the processes $\{\A^M,\tilde{h}^M\}$ converge almost surely and $R_k^M(x_1, -2^{-1/2}k^{1/2}x_2^{-1/2})$ converges almost surely for all $x_1, x_2 \in \Q^+$ along the subsequence $\{M_i\}_{i \in \Z_{\geq 1}}$. By Theorem \ref{thm_airy_line_ensemble}, the limit of $\A^M$ must be distributed as an Airy line ensemble, thus we denote it by $\A$. Let us also denote the limit of $\tilde{h}^M$ and $R_k^M(x_1, -2^{-1/2}k^{1/2}x_2^{-1/2})$ by $h$ and $R_k(x_1, -2^{-1/2}k^{1/2}x_2^{-1/2})$. The identity $h(0, \cdot) = \A_1(\cdot)$ follows from the definition of $\tilde{h}^M$.

By Lemma \ref{lemma_F_to_Airy}, we know that
\begin{equation}
R_k(x_1,-2^{-1/2}k^{1/2}x_2^{-1/2}) \dq \A[(0,k+1) \xrightarrow{\infty} (x_1, 1)] - 2^{3/2}k^{1/2}x_1^{1/2}.
\end{equation}
Moreover, by \cite[Theorem 6.3]{dauvergne2022directed}, we know that for any $\epsilon > 0$,
\begin{equation}
\sum_{k=1}^\infty \PP(|R_k(x_1,-2^{-1/2}k^{1/2}x_2^{-1/2})| > \epsilon k^{1/2}) < \infty.
\end{equation}
Hence, by Borel-Cantelli lemma, we conclude that $|k^{-1/2}R_k(x_1, -2^{-1/2}k^{1/2}x_2^{-1/2})|< \epsilon$ infinitely often almost surely. Therefore, with probability one,
\begin{equation}
\lim_{k \rightarrow \infty} |k^{-1/2}R_k(x_1, -2^{-1/2}k^{1/2}x_2^{-1/2})| =  0.
\end{equation}
\end{proof}

\begin{thm}\label{thm_airy_sheet}
    Let $h$ be any distributional limit of $\tilde{h}^N$ along some subsequence. Then there exists a coupling of $h$ and the Airy line ensemble $\A$ such that the following holds:
    \begin{enumerate}
        \item $h(0,\cdot) = \A_1(\cdot)$
        \item $h(\cdot + t, \cdot + t) \dq h(\cdot, \cdot)$ 
        \item Almost surely, for all $x > 0$ and $y_1, y_2 \in \R$, we have
        \begin{equation}
        \begin{split}
            &\lim_{k \rightarrow \infty} \A[(-2^{1/2}k^{1/2}x^{-1/2},k) \xrightarrow{\infty} (y_2,1)] - \A[(-2^{1/2}k^{1/2}x^{-1/2},k) \xrightarrow{\infty} (y_1,1)] =h(x,y_2) - h(x,y_1)
            \end{split}
        \end{equation}
    \end{enumerate}
    Then, $h$ has the same law as the Airy sheet. Hence $\tilde{h}^N$ converges in distribution to the Airy sheet.
\end{thm}

\begin{proof}
Let $h$ be the distributional limit of $h^N$ along some sequence $\{N_i\}_{i \in \Z_{\geq 1}}$. From Proposition \ref{coupling}, we can find a sequence $\{M_i\}_{i \in \Z_{\geq 1}}$ such that the assertions in Proposition \ref{coupling} hold. We can further augment the probability space to accomodate an Airy sheet $\SSS$ such that on an event with probability one,
\begin{equation}\label{Airy Sheet}
    \begin{split}
            &\lim_{k \rightarrow \infty} \A[(-2^{1/2}k^{1/2}x^{-1/2},k) \xrightarrow{\infty} (y_2,1)] - \A[(-2^{1/2}k^{1/2}x^{-1/2},k) \xrightarrow{\infty} (y_1,1)] =\SSS(x,y_2) - \SSS(x,y_1)
            \end{split}
\end{equation}
for all $x > 0$ and $y_1,y_2 \in \R$. Fix the event $\Omega_0$ with probability one such that for any $\omega \in \Omega_0$, all the assertions in Proposition \ref{coupling} and (\ref{Airy Sheet}) hold. We will show that when $\Omega_0$ occurs,
    \begin{equation}
    h(x,y_2) - h(x,y_1) = \SSS(x,y_2) - \SSS(x,y_1) 
    \end{equation}
    for all $x > 0$ and $y_1,y_2 \in \R$.

    Fix $x_1 < x_2 \in \Q^{+}$ and $y_1 \leq y_2 \in \Q$. We want to show that 
    \begin{equation}
    h(x_2,y_2) - h(x_2,y_1) \geq \SSS(x_1,y_2) - \SSS(x_1,y_1).
    \end{equation}
    Let $z_k = -2^{-1/2}k^{1/2}x_1^{-1/2}$. By Corollary \ref{h_S}, we have
    \begin{equation}\label{h_G}
    G_k^N(z_k,y_1)- G_k^N(z_k,y_2) + h^N(x_2,y_2) - h^N(x_2,y_1) \geq  \log (1-B_k^N(\overline{x}_{2,N},\widehat{y}_{1,N};\widehat{z}_{k,{N}}-1))    
    \end{equation}
    Because $\tilde{h}^N$ is the linear interpolation of $h^N$, if $\tilde{h}^N$ converges to some function $h$ uniformly on compact subsets, then so does $h^N$. Thus,
    \begin{equation}
    \lim_{k \rightarrow \infty}\lim_{\substack{N \in \{M_i\}\\ i \rightarrow \infty}}(\text{LHS of }(\ref{h_G}) ) = \SSS(x_1,y_1)- \SSS(x_1,y_2) + h(x_2,y_2) - h(x_2,y_1).
    \end{equation}
    Hence it suffices to show that 
    \begin{equation}
     \lim_{k \rightarrow \infty}\lim_{\substack{N \in \{M_i\}\\ i \rightarrow \infty}}(\log B_k^N(\overline{x}_{2,N},\widehat{y}_{1,N}; \widehat{z}_{k,{N}})) = - \infty.
    \end{equation}
    This follows by applying Corollary \ref{hR} and Proposition \ref{coupling}. A similar argument yields
    \begin{equation}
    h(x_2,y_2) - h(x_2,y_1) \leq \SSS(x_3,y_2) - \SSS(x_3,y_1)
    \end{equation}
    for all $x_2 < x_3 \in \Q^{+}$ and $y_1 \leq y_2 \in \Q$. By continuity, 
    \begin{equation}
    h(x,y_2) - h(x,y_1) = \SSS(x,y_2) - \SSS(x,y_1)
    \end{equation}
    for all $x > 0$ and $y_1, y_2 \in \R$.
    
    Finally, we verify the stationarity property $h(\cdot + t, \cdot + t) \dq h(\cdot,\cdot)$. This follows from the corresponding property for the prelimit objects: as processes on the discrete lattice $N^{-2/3}q^{2}\Z \times N^{-2/3}q^{2}\Z$, we have \begin{equation}\label{equation_invariance}
    h^N(\cdot + t(N),\cdot + t(N)) \dq h^N(\cdot, \cdot)
    \end{equation}
    for $t(N) \in N^{-2/3}q^{2}\Z$. 
    Since the mesh size $N^{-2/3}q^2$ vanishes as $N \rightarrow \infty$, we can find a sequence $\{t(N)\}_{N \in \N} \subset N^{-2/3}q^{2}\Z$ such that $t(N) \rightarrow t$. Taking limit on both sides of (\ref{equation_invariance}) yields the desired result.
\end{proof}

\section{Convergence to the directed landscape}
Recall the notation $\R^4_+ = \{(x,s;y,t) \in \R^4: s < t\}$ for the domain of time-ordered space-time points.

\begin{define}\label{def_lg_landscape}
    For $t, x \in \R$, we define the scaling operations $\overline{x}_N = \lfloor N^{2/3}xq^{-2} \rfloor +1$, $t_N = \lfloor 2Nt \rfloor$. Given $(x,s;y,t) \in \R_+^4$, we define the unscaled and scaled log-gamma landscape as follows:
    \begin{equation}
    \begin{split}
        \h^N(x,s;y,t) &= \log Z[(\overline{x}_N + s_N, s_N) \rightarrow (\overline{y}_N + t_N-1,t_N-1)] -p(\overline{y}_N - \overline{x}_N + 4N(t-s)),
    \end{split}
    \end{equation}
    \begin{equation}
    \begin{split}
    h^N(x,s;y,t)= 2^{-1/2}q\sigma_pN^{-1/3} \h^N(x,s;y,t).
    \end{split}
    \end{equation}
\end{define}

\subsection{Tightness}
Let $\tilde{h}^N(x,s;y,t)$ be the linear interpolated $h^N(x,s;y,t)$ such that they agree on tuples $(x,s;y,t) \in \R^4_+$ where $x,y \in N^{-2/3}q^2\Z$ and $s,t \in 2^{-1}N^{-1} \Z$. The goal of this section is to prove the tightness of $\tilde{h}^N(x,s;y,t)$ in $C(\R^4_+,\R)$. To achieve this goal, we will use the following lemma by Dauvergne and Virag \cite[Lemma 3.3]{dauvergne2021bulk}.

\begin{lemma}\label{DV21}
    Let $Q = I_1 \times \cdots \times I_d$ denote the Cartesian product of real intervals with lengths $L_1, \ldots, L_d$. Let $\mathcal{G}:Q \rightarrow \R$ be a continuous random function. Let $e_i$ be the $i$-th coordinate vector of $\R^d$. Suppose that for some $a,c >0$ and for each $i \in \{1, 2, \ldots, d\}$, there exist constants $\alpha_i \in (0,1)$, $\beta_i, r_i > 0$ such that for all $m>0$, $u \in (0,r_i)$ and $t,t+ue_i \in Q$, the following inequality holds:
\begin{equation}
\mathbb{P}\left(|\mathcal{G}(t + u e_i) - \mathcal{G}(t)| \geq m u^{\alpha_i}\right) \leq c e^{-a m^{\beta_i}}.
\end{equation}
Let $\beta = \min_i \beta_i$, $\alpha = \max_i \alpha_i$, and $r = \max_i r_i^{\alpha_i}$. Then, with probability one, for all $t, t+s \in Q$ such that $|s_i| \leq r_i$ for all $i$, we have

\begin{equation}
|\mathcal{G}(t + s) - \mathcal{G}(t)| \leq C \left(\sum_{i=1}^{d} |s_i|^{\alpha_i} \log^{1/\beta_i} \left(\frac{2 r^{1/\alpha_i}}{|s_i|}\right)\right),
\end{equation}
where $C$ is a random constant with the tail bound:
\begin{equation}
\mathbb{P}(C > m) \leq \left(\prod_{i=1}^{d} \frac{b_i}{r_i}\right) c c_0 e^{-c_1 m^{\beta}},
\end{equation}
for some constants $c_0, c_1 > 0$ depending only on on $\{\alpha_i\}, \{\beta_i\}, d$ and $a$.
\end{lemma}

Because $\tilde{h}^N$ is the linear interpolated version of $h^N$, it is difficult to directly prove the tail bound for $\tilde{h}^N$ as required in Lemma \ref{DV21}. However, the following proposition justifies the sufficiency to only prove the tail bound for the integer points of $h^N$.

\begin{prop}\label{prop_discrete_to_ctns}
    Suppose that there exists positive constants $C_1, C_2, r$ and $N_0 \in \N$ such that for all $N \geq N_0$, $d_1 \in (0,r] \cap N^{-2/3}q^2 \Z$, $d_2 \in (0,r] \cap 2^{-1}N^{-1} \Z$, $x,y \in N^{-2/3}q^2 \Z$, $s,t \in 2^{-1}N^{-1}\Z$, $s < t$, and every $K \geq 0$, we have
    \begin{equation}\label{equation_tail}
    \PP(|h^N(x,s;y,t+d_2) - h^N(x,s;y,t)| \geq Kd_2^{1/3}) \leq C_1e^{-C_2K^{1/10}}
    \end{equation}
    \begin{equation}
    \PP(|h^N(x,s+d_2;y,t) - h^N(x,s;y,t)| \geq Kd_2^{1/3}) \leq C_1e^{-C_2K^{1/10}}
    \end{equation}
    \begin{equation}
    \PP(|h^N(x,s;y+d_1,t) - h^N(x,s;y,t)| \geq Kd_1^{1/3}) \leq C_1e^{-C_2K^{1/10}}
    \end{equation}
    \begin{equation}
    \PP(|h^N(x+d_1,s;y,t) - h^N(x,s;y,t)| \geq Kd_1^{1/3}) \leq C_1e^{-C_2K^{1/10}}.
    \end{equation}
    Then all of the inequalities naturally extend to $\tilde{h}^N$ with different constants $\tilde{C}_1, \tilde{C}_2$.
\end{prop}

\begin{proof}
    We will prove the extension of (\ref{equation_tail}) and the rest will follow analogously. Let us first assume that $x,y \in N^{-2/3}q^2\Z$ and $s \in 2^{-1}N^{-1}\Z$. We want to bound the following probability
    \begin{equation}\label{equation_discrete_h}
    \PP(|\tilde{h}^N(x,s;y,t_2) - \tilde{h}^N(x,s;y,t_1)| \geq K|t_1 - t_2|^{1/3}).
    \end{equation}
    We first consider the case when $\lfloor 2Nt_1 \rfloor = \lfloor 2Nt_2 \rfloor$. In this setting, we let $w_1 = \lfloor 2Nt_1 \rfloor$ and $w_2 = \lceil 2Nt_1 \rceil$. Since $\tilde{h}^N$ is a linear interpolation of $h^N$,
    \[
    |\tilde{h}^N(x,s;y,t_2) - \tilde{h}^N(x,s;y,t_1)| = |h^N(x,s;y,w_1) - h^N(x,s;y,w_2)|\frac{|t_1 - t_2|}{|w_1 - w_2|}.
    \]
    Thus, the probability (\ref{equation_discrete_h}) is bounded by
    \[
    \begin{split}
        &\PP\left(|\tilde{h}^N(x,s;y,w_2) - \tilde{h}^N(x,s;y,w_1)| \geq K \left(\frac{|w_1-w_2|}{|t_1 - t_2|}\right)^{2/3}|w_1 - w_2|^{1/3}\right)\\
        &\leq \PP(|\tilde{h}^N(x,s;y,w_2) - \tilde{h}^N(x,s;y,w_1)| \geq K|w_1 - w_2|^{1/3})\\
        & \leq C_1e^{-C_2K^{1/10}}.
    \end{split}
    \]
    For the remaining case, we may assume without loss of generality that $\lfloor 2Nt_1 \rfloor < \lfloor 2Nt_2 \rfloor$. In this case, set $w_1 = \lceil 2Nt_1 \rceil$ and $w_2 = \lfloor 2Nt_2 \rfloor$. Thus, the probability (\ref{equation_discrete_h}) is bounded by
    \[
    \begin{split}
        &\PP(|\tilde{h}^N(x,s;y,t_2) - \tilde{h}^N(x,s;y,w_2)| \geq \frac{1}{3}K|w_2 - t_2|^{1/3})\\
        & \quad + \PP(|\tilde{h}^N(x,s;y,t_1) - \tilde{h}^N(x,s;y,w_1)| \geq \frac{1}{3}K|t_1 - w_1|^{1/3})\\
        & \quad + \PP(|\tilde{h}^N(x,s;y,w_2) - \tilde{h}^N(x,s;y,w_1)| \geq \frac{1}{3}K|w_1 - w_2|^{1/3}).
    \end{split}
    \]
    Here, the first two terms can be bounded using the argument from the first case, while the last term is controlled directly via equation (\ref{equation_tail}). 

    We now proceed to relax the constraint on the remaining coordinates one at a time. Suppose $x \in N^{-2/3}q^2 \Z$, $s \in 2^{-1}N^{-1}\Z$, and $y \in \R$. Let $y = \lambda y_1 + (1-\lambda)y_2$ for some $\lambda \in [0,1)$ and $y_1, y_2 \in N^{-2/3}q^2 \Z$. Then,
    \[
    \tilde{h}^N(x,s;y,t_1) = \lambda \tilde{h}^N(x,s;y_1,t_1) + (1-\lambda)\tilde{h}^N(x,s;y_2,t_1)
    \]
    \[
    \tilde{h}^N(x,s;y,t_2) = \lambda \tilde{h}^N(x,s;y_1,t_2) + (1-\lambda)\tilde{h}^N(x,s;y_2,t_2).
    \]
    Thus, the probability (\ref{equation_discrete_h}) can be bounded by the following
    \begin{equation}
        \begin{split}
         &\PP(\lambda|\tilde{h}^N(x,s;y_1,t_1) - \tilde{h}^N(x,s;y_1,t_2)| \geq \lambda K|t_1 - t_2|^{1/3})\\
        & \quad + \PP((1-\lambda)|\tilde{h}^N(x,s;y_2,t_1) - \tilde{h}^N(x,s;y_2,t_2)| \geq (1-\lambda)K|t_1 - t_2|^{1/3}). 
        \end{split}
    \end{equation}
    It then follows from the above argument. The constraints for the remaining coordinates can be relaxed in this inductive manner.
\end{proof}

In view of Proposition \ref{prop_discrete_to_ctns}, it suffices to establish the desired tail bounds at lattice points. In this setting, we can exploit the following symmetries arising from the i.i.d. environment of the log-gamma polymer.

\begin{lemma}\label{lemma_symmetry}
    Let $\text{Lattice}_N = \{(x,s;y,t)| s < t,s,t \in 2^{-1}N^{-1}\Z, x,y \in N^{-2/3}q^2\Z\}$. As a random function from $\text{Lattice}_N$ to $\R$, $h^N$ has the following symmetries: for any $s,t,a \in 2^{-1}N^{-1}\Z$ and $x,y,b \in N^{-2/3}q^{2}\Z$,
    \begin{equation}
    h^N(x+b,s+a;y+b,t+a) \dq h^N(x,s;y,t)
    \end{equation}
    \begin{equation}
    h^N(x,s;y,t) \dq h^N(-y,-t;-x,-s).
    \end{equation}
\end{lemma}

Consequently, it suffices to establish only the following two probability bounds, as the remaining cases in Lemma \ref{DV21} follow by symmetry.
\begin{prop}\label{prop_tail_bounds}
    Fix $t_0 > 0$ and $M >0$. There exists positive constants $C_1, C_2,r_0 > 0$ and $N_0 \in \Z_{\geq 1}$ such that for any $N \in \Z_{\geq N_0}$, $d_1 \in (0,r_0] \cap 2^{-1}N^{-1}\Z$ and $d_2 \in (0,r_0] \cap N^{-2/3}q^2\Z$, $K \geq 0$, $y \in N^{-2/3}q^2\Z, |y| \leq M$, and $t \in 2^{-1}N^{-1}\Z \cap (t_0, \infty)$, we have
    \begin{equation}
    \PP(|h^N(0,0;y;t+d_1) - h^N(0,0;y,t)| \geq Kd_1^{1/3}) \leq C_1e^{-C_2K^{1/10}}
    \end{equation}
    \begin{equation}
    \PP(|h^N(0,0;y+d_2;t) - h^N(0,0;y,t)| \geq Kd_2^{1/3}) \leq C_1e^{-C_2K^{1/10}}.
    \end{equation}
\end{prop}

We defer the proof of this proposition to the next section and proceed under the assumption that it holds.

\begin{thm}
    $\tilde{h}^N$ is tight in $C(\R_+^4, \R)$.
\end{thm}

\begin{proof}
Fix $b \in \Z_{\geq 1}$ and define  \begin{equation}
 Q_b = [-b,b]^4 \cap \{(x,s;y,t) \in \R^4: t-s \geq b^{-1}\}.
 \end{equation} 
 Because $Q_b$ is compact, $Q_b$ can be written as the union of finitely many hypercubes $Q_{b,i}$ where 
 \begin{equation}
 Q_{b,i} = \{(x,s;y,t) \in Q_b: (x,s;y,t) = (x_{b,i},s_{b,i};y_{b_i},t_{b_i}) + [0,r_0]^4 \}.
 \end{equation}
 To show that $\tilde{h}^N$ is tight in $C(\R_+^4, \R)$, it suffices to show that $\tilde{h}^N$ is tight in $Q_{b,i}$ because any compact subsets of $\R^4_+$ will be contained in some $Q_b$ for some large enough $b$. By the tightness of the log-gamma line ensemble in \cite[Theorem 1.10]{barraquand2023spatial}, the sequence $\{h^N(x(N),s(N);y(N),t(N))\}_{N \geq 1}$ is tight whenever $(x(N),s(N);y(N),t(N)) \rightarrow (x,s;y,t)$ as $N \rightarrow \infty$. Since each $\tilde{h}^{N}(x_{b,i},s_{b,i};y_{b_i},t_{b_i})$ is defined via linear interpolation of finitely many $h^{N}(x(N_j),s(N_j);y(N_j),t(N_j))$ where $(x(N_j),s(N_j);y(N_j),t(N_j)) \rightarrow (x_{b,i},s_{b,i};y_{b_i},t_{b_i})$, the seqeunce $\{\tilde{h}^{N}(x_{b,i},s_{b,i};y_{b_i},t_{b_i})\}_{N \geq 1}$ is also tight. Thus, it remains to develop the modulus of continuity control.

To apply Proposition \ref{prop_tail_bounds}, we take $t_0 = b^{-1}$ and $M=b$. Together with Lemma \ref{lemma_symmetry} and Proposition \ref{prop_discrete_to_ctns}, we could directed apply Lemma \ref{DV21} and conclude that almost surely for any $(x_1, s_1;y_1,t_1), (x_2,s_2;y_2,t_2) \in Q_{b,i}$,
\[
\begin{split}
    |\tilde{h}^N(s_1,x_1;y_1,t_1) - \tilde{h}^N(s_2,x_2;y_2,t_2)|
    &\leq C^N\bigg(|x_1-x_2|^{1/3}\log^{10}(2r_0/|x_1-x_2|) +|y_1-y_2|^{1/3}\log^{10}(2r_0/|y_1-y_2|)\\
    &\quad +|t_1-t_2|^{1/3}\log^{10}(2r_0/|t_1-t_2|) +|s_1-s_2|^{1/3}\log^{10}(2r_0/|s_1-s_2|)\bigg)
\end{split}
\]
where $\PP(C^N > m) \leq De^{-Dm^{1/10}}$ for some constant $D$ that depends on $b$ only. Therefore, by the Kolmogorov-Chentsov criterion, we see that $\tilde{h}^N$ is tight in $Q_{b,i}$.
\end{proof}

\subsection{Convergence}
Given the tightness of $\tilde{h}^N$, we now introduce the following variational formula, an analogue of Lemma \ref{Bisection}, which will serve as a key tool in establishing the convergence of $\tilde{h}^N$ to the directed landscape.

\begin{lemma}
    For $\ell,k,m \in \Z$ such that $\ell < k \leq m$,
    \begin{equation}\label{equation_variational}
    Z[(x,\ell) \rightarrow (y,m)] = \sum_{i=x}^y Z[(x,\ell) \rightarrow (i,k+1)]Z[(i,k) \rightarrow (y,m)].
    \end{equation}
        Thus, for $s < r < t$
\begin{equation}
\exp( \h^N(x,s;y,t) )=  \sum_{\overline{z}_N = \overline{x}_N+s_N}^{\overline{y}_N + t_N-1} \exp(\h^N(x,s; z,r) + \h^N(z,r;y,t))
\end{equation}
\begin{equation}\label{equation5}
\begin{split}
&h^N(x,s;y,t)= 2^{-1/2}q\sigma_pN^{-1/3} \log \sum_{\overline{z}_N = \overline{x}_N+s_N}^{\overline{y}_N + t_N-1} \exp(2^{1/2}q^{-1}\sigma_p^{-1}N^{1/3}[h^N(x,s; z,r) + h^N(z,r;y,t)]).  
\end{split}
\end{equation}    
\end{lemma}

\begin{proof}
    Observe that equation (\ref{equation_variational}) follows directly from summing over all possible intermediate points where an up-right path $\pi$ from $(x,\ell)$ to $(y,m)$ may intersect the $i$-th row and the rest follows from Definition \ref{def_lg_landscape}.
\end{proof}

\begin{thm}
    $\tilde{h}^N$ converges to $\mathcal{L}$ in  distribution as $C(\R^4_+,\R)$-random variables.
\end{thm}

\begin{proof}
    Since $\tilde{h}^N$ is tight, let $h$ denote a distributional limit along some subsequence $\{N_i\}_{i \in \N}$. By the Skorokhod's representation theorem, there exists a coupling under which $\tilde{h}^{N_i}$ converges to $h$ almost surely in $C(\R^4_+, \R)$. Let $\Omega_0$ denote the event on which this convergence holds; then $\PP(\Omega_0)=1$. For any finite set of disjoint intervals $\{(s_j,t_j)\}_{j=1}^m$, $\{h(\cdot,s_j;\cdot,t_j)\}_{j=1}^m$ are independent because $\{\tilde{h}^{N_i}(\cdot,s_j;\cdot,t_j)\}_{j=1}^m$ are independent.

    Moreover, since $\tilde{h}^{N_i} \rightarrow h$ on $\Omega_0$ in $C(\R^4_+, \R)$, it follows that  $h^{N_i} \rightarrow h$ uniformly over compact subsets on $\Omega_0$ as well. For any $r > 0$, to show that the process $h(\cdot,s;\cdot,s+r^3)$ has the law of Airy sheet of scale $r$,  consider sequences $\{s_{i}\}$ and $\{r_i\}$ such that $2N_is_i, 2N_ir_i^3 \in \Z$ and $(s_i,r_i) \rightarrow (s,r)$ as $i \rightarrow \infty.$ Notice that $h^{N_i}(\cdot,s_i;\cdot,s_i+r_i^3) \rightarrow h(\cdot,s;\cdot,s+r^3)$ on $\Omega_0$. Let $M_i = N_ir_i^3$ and take $M_i$ to infinity. Then by Theorem \ref{thm_airy_sheet}, we know that $h(\cdot,s;\cdot,s+r^3)$ must be distributed as an Airy sheet of scale $r$.

    The only thing left to prove is that for any $t_i < t_j < t_k$, $x,y \in \R$, the following holds with probability one:
\begin{equation}\label{equation3}
h(x, t_i; y, t_k) = \max_{z \in \mathbb{R}} \left( h(x, t_i; z, t_j) + h(z, t_j; y, t_k) \right).
\end{equation}

From \cite[Proposition 9.2]{dauvergne2022directed}, we know that the right-hand side of (\ref{equation3}) is well-defined as a random variable on $C(\mathbb{R}^2, \mathbb{R})$ and is distributed as an Airy sheet of scale $(t_k - t_i)^{1/3}$. Since $h(\cdot,t_i;\cdot,t_k)$ is also distributed as an Airy sheet of scale $(t_k - t_i)^{1/3}$, it is enough to show that almost surely for all $x, y \in \mathbb{R}$,
\begin{equation}\label{equation4}
h(x, t_i; y, t_k) \geq \max_{z \in \mathbb{R}} \left( h(x, t_i; z, t_j) + h(z, t_j; y, t_k) \right).
\end{equation}

Let $\Omega_1$ denote the event on which the right-hand side of (\ref{equation3}) defines a continuous function in $x$ and $y$. Then $\PP(\Omega_1) = 1$. Let us denote a maximizer of the function $h(x, t_i; z, t_j) + h(z, t_j; y, t_k)$ by $Z_j(x, t_i; y, t_k)$. Note that on the event $\Omega_0\cap \Omega_1$, $Z_j(x, t_i; y, t_k) \neq \emptyset$. For $M > 0$, consider the event $\Omega_0 \cap \Omega_1 \cap \{Z_j(x, t_i; y, t_k) \cap [-M, M] \neq \emptyset\}$. When such an event occurs, we have
\begin{equation}
\begin{split}
&\max_{z \in \mathbb{R}} \left( h(x, t_i; z, t_j) + h(z, t_j; y, t_k) \right)\\
&= \max_{z \in [-M, M]} \left( h(x, t_i; z, t_j) + h(z, t_j; y, t_k) \right) \\
&= \lim_{i \to \infty} 2^{-1/2}q\sigma_pN_i^{-1/3}  \log \int_{-M}^M \exp \left[ 2^{1/2}q^{-1}\sigma_p^{-1}N_i^{1/3}  (h^{N_i}(x, t_i; z, t_j) + h^{N_i}(z, t_j; y, t_k)) \right]dz  \\
&\leq \lim_{i \to \infty} 2^{-1/2}q\sigma_pN_i^{-1/3}  \log \sum_{\substack{z \in [-M-1,M]\\ \overline{z}_{N_i} = \overline{x}_{N_i} + t_{i,N_i}}}^{ \overline{y}_{N_i} + t_{k,N_i}-1}  {N_i^{-2/3}q^2}\exp \left[ 2^{1/2}q^{-1}\sigma_p^{-1}N_i^{1/3}  (h^{N_i}(x, t_i; z, t_j) + h^{N_i}(z, t_j; y, t_k)) \right]  \\
&\leq \lim_{i \to \infty} h^{N_i}(x, t_i; y,t_k)\\& = h(x, t_i; y,t_k).
\end{split}    
\end{equation}
The second equality follows from the Laplace method. The first inequality follows from the definition of $h^{N_i}$, and the second inequality is a consequence of  equation (\ref{equation5}). Since $\PP(\Omega_0 \cap \Omega_1) = 1$ and $\lim_{M \rightarrow \infty} \Omega_0 \cap \Omega_1 \cap \{Z_j(x, t_i; y, t_k) \cap [-M, M] \neq \emptyset\} = \Omega_0\cap \Omega_1$, we conclude that (\ref{equation4}) holds with probability one.
\end{proof}

\section{Proof of Proposition \ref{prop_tail_bounds}}
In this section, we focus on the proof of Proposition \ref{prop_tail_bounds}. Throughout the section, we assume that all coordinates under consideration are integer-valued. Since we will frequently refer to lemmas and propositions from \cite{basu2024temporal}, we introduce the following notations to keep things consistent.

For $u,v \in \Z^2$, let $Z_{u,v} = Z[u \rightarrow v]$ and let $\ZZ_{u,v} = \frac{Z_{u,v}}{d_u}$.

For $(N,N),(M,M) \in \Z^2$, we abbreviate $Z_{(N,N),(M,M)}$ and $\ZZ_{(N,N),(M,M)}$ as $Z_{N,M}$ and $\ZZ_{N,M}$.

For any $w \in \Z^2$, define the anti-diagonal through $w$ as $L_w = \{w+(i,-i): i \in \Z\}$ and its truncation by $k \in \Z_{\geq 0}$ as $L_w^k = \{x \in L_w: ||x - w||_\infty \leq k\}$.

For $A,B \subset \Z^2$, let $Z_{A,B} = \sum_{a \in A, b \in B}Z_{a,b}$,  $Z_{A,B}^{\max} = \max_{a \in A,b \in B}Z_{a,b}$ and accordingly for $\tilde{Z}_{A,B}$ and $\tilde{Z}^{\max}_{A,B}$.

We define the characteristic direction of the polymer model as a function of $\rho \in (0, \theta)$ by
\begin{equation}
\xi[\rho] = \left(\frac{\Psi_1(\rho)}{\Psi_1(\rho)+\Psi_1(\theta-\rho)},\frac{\Psi_1(\theta - \rho)}{\Psi_1(\rho)+\Psi_1(\theta-\rho)}\right)
\end{equation}
where $\Psi_1$ denotes the trigamma function. Since $\Psi_1$ is strictly decreasing and smooth on $\R_{>0}$, the map $\xi$ defines a continuous bijection between $\rho \in (0, \theta)$ and vectors in the open line segment between $(1,0)$ and $(0,1)$.

Corresponding to each vector $\xi[\rho]$, we define the shape function as
\begin{equation}
\Lambda(\xi[\rho]) = - \frac{\Psi_1(\rho)}{\Psi_1(\rho) + \Psi_1(\theta-\rho)}\Psi(\theta-\rho) - \frac{\Psi_1(\theta-\rho)}{\Psi_1(\rho) + \Psi_1(\theta-\rho)}\Psi(\theta-\rho)
\end{equation}
where $\Psi$ is the digamma function.

Let $\Lambda(N,N) := 2N\Lambda(\xi[\frac{\theta}{2}]) = -2N\Psi(\frac{\theta}{2})$ to denote the shape function in the diagonal direction. Note that this value is consistent with our earlier definitions: $h_{\theta}(1) = 2\Psi(\frac{\theta}{2})$ and $p = -\Psi(\frac{\theta}{2})$. 

Now, for $w = (N + yN^{2/3}, N - yN^{2/3})$, we can express $w$ in terms of its characteristic direction as $w = 2N\xi[\frac{\theta}{2} + z_w]$ for some $z_w \in \R$. We then define the shape function at $w$ by $\Lambda(w) := 2N\Lambda(\xi[\frac{\theta}{2} + z_w])$.

\begin{prop}\label{prop_Line_to_point}
    Let $w = (N + yN^{2/3}, N - yN^{2/3})$. There exists constant $c_1,c_2, C_1,K_0 > 0$ and $N_0 \in \Z_{\geq 1}$ such that for $N \geq N_0$, $K_0 \leq K \leq c_1N^{1/2}$, $|y| \leq c_2K^{1/10}$ and $a \in \Z_{
    \geq 0}$, it holds that
\begin{equation}
\PP( \log Z_{0, L_w^a} - \log Z_{0,w} > K\sqrt{a}) \leq e^{-C_1 \min\{K^2, K\sqrt{a}\}}.
\end{equation}
\end{prop}

\begin{proof}

We will prove the statement by considering two separate cases. We begin with the case where $a > K^{2/3}N^{2/3}$. By \cite[Proposition 3.5]{basu2024temporal}, the following inequality holds 
\begin{equation}
|\Lambda(w) - \Lambda(N,N)| \leq Cy^2 N^{1/3}.
\end{equation}
for some constant $C > 0$.
Thus, we can bound the difference between the free energy using the one-point estimate:
\begin{equation}
\begin{split}
&\PP(\log Z_{0, L_w^a} - \log Z_{0,w} > K \sqrt{a})\\
& \leq \PP(\log \ZZ_{0, L_N} - \log \ZZ_{0,w} > K^{4/3}N^{1/3})\\
& \leq \PP(\log \ZZ_{0, L_N} - \Lambda(w) > \frac{1}{2} K^{4/3}N^{1/3}) + \PP(\log \ZZ_{0, w} - \Lambda(w) < - \frac{1}{2} K^{4/3}N^{1/3}).
\end{split}
\end{equation}
The second term can be directly bounded by \cite[Proposition 3.8]{basu2024temporal}. For the first term, we can apply the bound between the shape function:
\begin{equation}
\begin{split}
    \PP(\log \ZZ_{0, L_N} - \Lambda(w) > \frac{1}{2} K^{4/3}N^{1/3})
    &= \PP(\log \ZZ_{0, L_N} - \Lambda(N,N) + \Lambda(N,N) - \Lambda(w) > \frac{1}{2} K^{4/3}N^{1/3})\\
    & \leq \PP(\log \ZZ_{0, L_N} - \Lambda(N,N) > \frac{1}{2}K^{4/3}N^{1/3} - Cy^2N^{1/3}).
\end{split}
\end{equation}
Now, choose $c_2 \leq \frac{1}{\sqrt{4C}}$. Then, under the assumption that $|y| \leq c_2K^{1/10}$, we have 
\begin{equation}\label{equation_p}\PP(\log \ZZ_{0, L_N} - \Lambda(N,N) > \frac{1}{2}K^{4/3}N^{1/3} - Cy^2N^{1/3}) \leq \PP(\log \ZZ_{0, L_N} - \Lambda(N,N) >  \frac{1}{4}K^{4/3}N^{1/3}).\end{equation} The right-hand side of equation (\ref{equation_p}) can then be upper bounded according to \cite[Proposition 3.6]{basu2024temporal}.

The case where $a \leq K^{2/3}N^{2/3}$ follows from a similar random walk approximation based on the stationary polymer measure, as in \cite[Proposition 4.1]{basu2024temporal}. Notice that 
\begin{equation}
\log Z_{0, L_w^a} \leq \log Z^{\max}_{0, L_w^a} +  \log (2a+1).
\end{equation}

It is sufficient to bound the following probability for some constant $C'$:
\begin{equation}
\mathbb{P}\left(\log Z_{0, L_{w}^{a,+}}^{\max} - \log Z_{0,w} \geq C'K \sqrt{a} \right) + \mathbb{P}\left(\log Z_{0, L_{w}^{a,-}}^{\max} - \log Z_{0,w} \geq C'K \sqrt{a} \right)
\end{equation}
where $\mathcal{L}_{w}^{a,+}$ and $\mathcal{L}_{w}^{a,-}$ denotes the subset of $\mathcal{L}_{w}^{a}$ lying to the left and above $w$, and to the right and below $w$, respectively. We will prove the bound for the first term as the bound for the second term is completely analogous. For any fixed $k = 0, \ldots, a$, we redefine the difference of free energy along the anti-diagonal as:
\begin{equation}
\log Z_{0,w+(-k,k)} - \log Z_{0,w} = \sum_{i=1}^{k} \log Z_{0,w+(-k+i-1,k-i+1)} - \log Z_{0,w+(-k+i,k-i)} = S_k.
\end{equation}
This reformulation allows us to study the behavior of the walk $S_k$ via its running maximum:
\begin{equation}
\mathbb{P}\left(\max_{0 \leq k \leq a} S_k \geq C' K \sqrt{a} \right).
\end{equation}

Although the increments of $S_k$ are neither independent nor identically distributed, \cite[Theorem 3.28]{basu2024temporal} provides a coupling with an i.i.d. random walk $\tilde{S}_k$ which upper bounds $S_k$ with high probability. Specifically, we take the down-right path $\Theta_{2a}$ to be the staircase from $w+(-a, a)$ to $w$ and define the perturbed parameter to be $\lambda =\frac{\theta}{2}+z_w + q_0 K^{2/3} N^{-1/3}$. Under this perturbation, the increments of $\tilde{S}_k$ are i.i.d. with distribution $\log(\text{Ga}^{-1}(\theta -\lambda)) - \log(\text{Ga}^{-1}(\lambda))$.

Let $A$ denote the event that $\log \frac{10}{9} + \tilde{S}_k \geq S_k$ for all $k \in \lb 0, a \rb$. Then we have the bound:
\begin{align*}
\mathbb{P}\left(\max_{0 \leq k \leq a} S_k \geq C' \sqrt{a} K^{3/4} \right) &\leq \mathbb{P}\left(\left\{\max_{0 \leq k \leq a} S_k \geq C' \sqrt{a} K^{3/4}\right\} \cap A \right) + \mathbb{P}(A^c) \\
&\leq \mathbb{P}\left(\left\{\log \frac{10}{9} + \max_{0 \leq k \leq a} \tilde{S}_k \geq C' \sqrt{a} K^{3/4}\right\}\right) + \mathbb{P}(A^c).
\end{align*}

To apply \cite[Theorem 3.28]{basu2024temporal}, we require $N \geq N_0$,  $K_0^{2/3} \leq K^{2/3} \leq c_1^{2/3}N^{1/3}$, and $1 \leq a \leq K^{2/3}N^{2/3}$ for some positive constants $N_0, K_0, c_1$. With all the requirements satisfied by our assumptions, we know that $\mathbb{P}(A^c) \leq e^{-C'' K^2}$ for some constant $C'' >0$. Absorbing the constant $\log(10/9)$ into the constant $C'$, it suffices to obtain the upper bound:
\begin{equation}
\mathbb{P}\left(\max_{0 \leq k \leq a} \tilde{S}_k \geq C' K \sqrt{a} \right) \leq e^{-C_1 \min\{K^2, K \sqrt{a}\}}.
\end{equation}

 This is a standard estimate on the running maximum of an i.i.d. random walk with sub-exponential increments, as proved in \cite[Appendix D]{basu2024temporal}. 
\end{proof}

\begin{prop}\label{prop_t^10}
    Let $w = (N + yN^{2/3}, N - yN^{2/3})$. There exists $N_0 \in \N_{\geq 0}$ and $C_2, K_0 > 0$ such that for each $N \geq N_0$, $K \geq K_0$, $a \in \Z_{\geq 0}$, and $|y| \leq c_2K^{1/10}$ where $c_2$ is the constant from Proposition \ref{prop_Line_to_point}, we have
    \begin{equation}
    \PP(\log Z_{0, L_w^a} - \log Z_{0, w} \geq K\sqrt{a}) \leq e^{-C_2K^{1/5}}.
    \end{equation}
\end{prop}

\begin{proof}
    Because of Proposition \ref{prop_Line_to_point}, we only need to establish the inequality in the regime where $K \geq c_1N^{1/2}$. Let $K = zN^{1/2}$ for some $z \geq c_1$. Then,
    \begin{equation}
    \begin{split}
        &\PP(\log Z_{0, L_{w}^a} - \log Z_{0,w} \geq K\sqrt{a})\\
        &\leq \PP(\log \ZZ_{0, L_{w}^a} - \log \ZZ_{0,w} \geq zN^{1/6}N^{1/3})\\
        &\leq \PP(\log \ZZ_{0, L_{w}^a}  - \Lambda(w) \geq \frac{1}{2}zN^{1/6}N^{1/3}) +  \PP( \log \ZZ_{0,w} -\Lambda(w) \leq -\frac{1}{2}zN^{1/6}N^{1/3})\\
        &\leq \PP(\log \ZZ_{0, L_{w}}  - \Lambda(N,N) \geq \frac{1}{2}zN^{1/6}N^{1/3}-Cy^2N^{1/3}) +  \PP( \log \ZZ_{0,w} -\Lambda(w) \leq -\frac{1}{2}zN^{1/6}N^{1/3})\\
        &\leq \PP(\log \ZZ_{0, L_{w}}  - \Lambda(N,N) \geq \frac{1}{4}zN^{1/6}N^{1/3}) +  \PP( \log \ZZ_{0,w} -\Lambda(w) \leq -\frac{1}{2}zN^{1/6}N^{1/3})\\
        &\leq e^{-C_2K^{1/5}}.
    \end{split}
    \end{equation}
We once again absorb the term $Cy^2N^{1/3}$ into $\frac{1}{2}zN^{1/6}N^{1/3}$ in the second to last inequality, using the bound $| y |\leq c_2 K^{1/10}$. The final inequality then follows from \cite[Proposition 3.8]{basu2024temporal} together with \cite[Proposition A.2]{basu2024temporal}.
\end{proof}

Let $0 \leq r \leq \frac{N}{2}$, and define $x^*$ be the random maximizer of the following across the anti-diagonal line $L_r$.
\begin{equation}
x^* = \arg \max_{x \in L_r}\{\log Z_{0,x}+\log Z_{x,w}\}.
\end{equation}

\begin{prop}\label{prop_maximizer}
There exist positive constants $C_3, c_3,c_4, c_5, M_0, N_0$ such that for each $N \geq N_0$, $c_3 \leq r \leq c_4(N-r)$, $M \geq M_0$, $| y| \leq  c_5M^{1/10}$, we have:
\begin{equation}
\mathbb{P}(||x^* - (r, r)||_\infty > M r^{2/3}) \leq e^{-C_{3} M^3}.
\end{equation}
    
\end{prop}
\begin{proof}
    
Let \( J_h = L^{r^{2/3}}_{(r-2 h r^{2/3}, r + 2 h r^{2/3})} \) denote the segment of the anti-diagonal line $L_r$ that consists of points within a window of width $4hr^{2/3}$ centered at $r$. We now bound the probability
\begin{equation}
\begin{split}
    &\PP\left( \left|| x^* - (r, r) \right||_\infty > M r^{2/3} \right)\\
    &\leq \PP\left( \max_{x \in L_r \setminus L_r^{Mr^{2/3}}} \{\log Z_{0,x} + \log Z_{x,w}\}  > \log Z_{0,r} + \log Z_{r,w} \right)\\
    & \leq \sum_{|h| = \lfloor M/2 \rfloor}^{r^{1/3}} \PP(\log Z_{0,J^h}^{\max} + \log Z^{\max}_{J^h,w} > \log Z_{0,r} + \log Z_{r,w})\\
    & \leq \sum_{|h| = \lfloor M/2 \rfloor}^{r^{1/3}} \PP(\log Z_{0,J^h}^{\max} - \log Z_{0,r} \geq -Dh^2r^{1/3}) + \PP( \log Z^{\max}_{J^h,w} - \log Z_{r,w}  \geq Dh^2r^{1/3})
\end{split}
\end{equation}
for some small positive constant $D$ to be chosen later.

For the first term in the summation, we use \cite[Propositions 3.8]{basu2024temporal} and \cite[Proposition 3.11]{basu2024temporal} to obtain:
\begin{equation}\label{equation_bound1}
\begin{split}
    \PP(\log Z_{0,J^h}^{\max} - \log Z_{0,r} \geq -Dh^2r^{1/3})
    &\leq \PP(\log \ZZ_{0,J^h}^{\max} - 2rp \geq -2Dh^2r^{1/3}) + \PP(\log \ZZ_{0,r} - 2rp \leq - Dh^2r^{1/3})\\
    &\leq e^{-C|h|^3}
\end{split}
\end{equation}
provided that \( 2D \leq C_{20} \), where $C_{20}$ is the constant from \cite[Proposition 3.11]{basu2024temporal}.

For the second term in the summation, we can upper bound it by
\begin{equation}
\begin{split}
    \PP( \log Z^{\max}_{J^h,w} -  \log Z_{r,w} \geq Dh^2r^{1/3}) &\leq \PP( \log Z^{\max}_{L_r^{4|h|r^{2/3}},w} -  \log Z_{r,w} \geq Dh^2r^{1/3})\\
   &\leq \PP( \log Z_{L_r^{4|h|r^{2/3}},w} -  \log Z_{r,w} \geq \frac{1}{2}D|h|^{3/2}\sqrt{4|h|r^{2/3}})\\
   &= \PP( \log Z_{L_{w'}^{4|h|r^{2/3}}} -  \log Z_{0,w'} \geq \frac{1}{2}D|h|^{3/2}\sqrt{4|h|r^{2/3}})
\end{split}
\end{equation}
where 
\begin{equation}
    w' = \left(N-r + y\left(\frac{N}{N-r}\right)^{2/3}(N-r)^{2/3}, N-r - y\left(\frac{N}{N-r}\right)^{2/3}(N-r)^{2/3}\right)
\end{equation} 
and the last equality follows from the symmetry of the partition function. We can now apply Proposition \ref{prop_Line_to_point} by setting $K = \frac{D}{2}|h|^{3/2}$ and $a = 4|h|r^{2/3}$. To apply Proposition \ref{prop_Line_to_point}, we need $\frac{D}{2}|h|^{3/2} \leq c_1(N-r)^{1/2}$ where $c_1$ is the constant from Proposition \ref{prop_Line_to_point}. Since $\frac{D}{2}|h|^{3/2} \leq \frac{Dr^{1/2}}{2}$, we can just take $c_4 = \frac{4c_1^2}{D^2}$. Since $\left(\frac{N}{N-r}\right)^{2/3} \leq (c_4 + 1)^{2/3}$, we also need $|y| \leq c_2K^{1/10} (c_4+1)^{-2/3} = C'|h|^{3/20}$. Thus, it suffices to require $|y| \leq c_5 M^{1/10}$ for some new constant $c_5 >0$.  Finally, we get that 
\begin{equation}\label{equation_bound2}
\PP( \log Z^{\max}_{J^h,w} -  \log Z_{r,w} \geq Dh^2r^{1/3}) \leq e^{-C_1K^2} \leq e^{-C|h|^3}.
\end{equation}
Combining the bounds from \ref{equation_bound1} and \ref{equation_bound2}, we obtain
\begin{equation}
\begin{split}
    &\PP\left( \left||x^* - (r, r) \right||_\infty > M r^{2/3} \right)\\
    & \leq \sum_{|h| = \lfloor M/2 \rfloor}^{r^{1/3}} \PP(\log Z_{0,J^h}^{\max} - \log Z_{0,r} \geq -Dh^2r^{1/3}) + \PP( \log Z^{\max}_{J^h,w} - \log Z_{r,w}  \geq Dh^2r^{1/3})\\
    &\leq \sum_{|h| = \lfloor M/2 \rfloor}^\infty e^{-C|h|^3}\\
    &\leq e^{-C_3M^3}.
\end{split}
\end{equation}
\end{proof}

\begin{thm}\label{thm_temporal_tail}
Let $w = (N + yN^{2/3}, N - yN^{2/3})$. There exists positive constants $C_4,K_0,c_6$ and $N_0 \in \Z_{\geq 1}$ such that for all $N \geq N_0, c_3 \leq r \leq c_4(N-r)$, $K \geq K_0$, and $| y| \leq c_6K^{1/20}$ where $c_3, c_4$ are constants from \ref{prop_maximizer}, we have
\begin{equation}
\PP(\log Z_{0,w} - \log Z_{0,r} - \log Z_{r,w} \geq Kr^{1/3}) \leq 2e^{-C_4K^{1/10}}
\end{equation}
\end{thm}
\begin{proof}
    Notice that it suffices to replace $\log Z_{0,w}$ by $\max_{x \in L_r}\{\log Z_{0,x} + \log Z_{x,w}\}$. This is because
    \begin{equation}
    \log Z_{0,w} \leq \max_{x \in L_r}\{\log Z_{0,x} + \log Z_{x,w}\} + \log(2r+1).
    \end{equation}
    Thus, we only need to consider the following probability
    \begin{equation}
    \PP(\max_{x \in L_r}\{\log Z_{0,x} + \log Z_{x,w}\} - \log Z_{0,r} - \log Z_{r,w} \geq Kr^{1/3}).
    \end{equation}
    We can split this into two cases by a union bound. The first case is that $x^* \notin L^{Kr^{2/3}}_r$ and the second case is that $x^* \in L^{Kr^{2/3}}_r$. By Proposition \ref{prop_maximizer}, we know that the first probability is bounded by $e^{-C_2K^3}$. For the case that $x^* \in L^{Kr^{2/3}}_r$, we have the following bound:
    \begin{equation}
    \begin{split}
        &\PP(\max_{x \in L^{Kr^{2/3}}_r}\{\log Z_{0,x} + \log Z_{x,w}\} - \log Z_{0,r} - \log Z_{r,w} \geq Kr^{1/3})\\
        &\leq \PP(\log Z_{0,L_r^{Kr^{2/3}}} - \log Z_{0,r} \geq \frac{1}{2}Kr^{1/3}) + \PP(\log Z_{L_r^{Kr^{2/3}},w} - \log Z_{r,w} \geq \frac{1}{2}Kr^{1/3})\\
        &\leq 2e^{-C_4K^{1/10}}
    \end{split}
    \end{equation}
    where the last inequality comes from Proposition \ref{prop_t^10}.
\end{proof}

Now we are ready to prove Proposition \ref{prop_tail_bounds}, which is restated below.
\begin{prop}
    Fix $t_0 > 0$ and $M > 0$. There exist positive constants $C_5, C_6,r_0$, $N_1 \in \Z_{\geq 1}$ for any $N \geq N_1$, $d_1 \in (0,r_0] \cap 2^{-1}N^{-1}\Z$ and $d_2 \in (0,r_0] \cap N^{-2/3}q^2\Z$, $K \geq 0$, $y \in N^{-2/3}q^2\Z, |y| \leq M$, and $t \in 2^{-1}N^{-1}\Z \cap (t_0, \infty)$, we have
    \begin{equation}\label{equation_temporal}
    \PP(|h^N(0,0;y;t+d_1) - h^N(0,0;y,t)| \geq Kd_1^{1/3}) \leq C_6e^{-C_5K^{1/10}}
    \end{equation}
    \begin{equation}\label{equation_spatial}
    \PP(|h^N(0,0;y+d_2;t) - h^N(0,0;y,t)| \geq Kd_2^{1/3}) \leq C_6e^{-C_5K^{1/10}}.
    \end{equation}
\end{prop}

\begin{proof}
    Let us begin by proving (\ref{equation_temporal}) first. We consider two separate cases.

        \textbf{Case 1}: $2Nd_1 \geq c_3$, where $c_3$ is the constant in Theorem \ref{thm_temporal_tail}. In this case, we can apply Theorem \ref{thm_temporal_tail} by requiring $N_1$ to be large enough so that $2N_1t_0 \geq N_0$ and $r_0$ small enough so that $r_0 \leq c_4t_0$ for constants $N_0$ and $c_4$ in Theorem \ref{thm_temporal_tail}. Lastly, we choose $K_0$ large enough so that $c_6K_0^{1/20} \geq q^2(2t_0)^{-2/3}M$ for constant $c_6$ in Theorem \ref{thm_temporal_tail}. Let $w_1 = (2Nt, 2Nt + N^{2/3}yq^{-2})$ and $w_2 = (2N(t+d_1), 2N(t+d_1) + N^{2/3}yq^{-2})$. Then,
    \[
    \begin{split}
            &\PP(|h^N(0,0;y,t+d_1)- h^N(0,0;y,t)|\geq Kd_1^{1/3})\\
            & = \PP(|\log Z_{1, w_2} - \log Z_{1, w_1}-4pNd_1|\geq Kd^{1/3}N^{1/3}2^{1/2}q^{-1}\sigma_p)\\
            & \leq \PP(|\log Z_{1, w_2} - \log Z_{1, w_1}-\log Z_{w_1,w_2}|\geq Kd_1^{1/3}N^{1/3}2^{-1/2}q^{-1}\sigma_p)\\
            & \quad + \PP(|\log Z_{w_1,w_2} - 4pNd_1|\geq Kd_1^{1/3}N^{1/3}2^{-1/2}q^{-1}\sigma_p)\\
        \end{split}
    \]
    Because $\log Z_{1,w_2} - \log Z_{1,w_1} \geq \log Z_{w_1,w_2}$ holds regardless of the random environment, we can remove the absolute value in the first term and apply Theorem \ref{thm_temporal_tail}. For the second term, it can be directly bounded by the one-point tail bound in \cite[Proposition 3.6]{basu2024temporal} and \cite[Proposition 3.8]{basu2024temporal}. The condition that $K \geq K_0$ for sufficiently large $K_0$ can be absorbed into the constant $C_6.$
        
    \textbf{Case 2}: $1 \leq 2Nd_1 \leq c_3$. For $K \leq a_0^{3/2}N^{1/2}$ where $a_0$ is the constant from \cite[Theorem 3.28]{basu2024temporal}, we apply the same random walk approximation twice as in the proof of Proposition \ref{prop_Line_to_point}. Let $w = ( 2N(t +d_1) , 2N(t +d_1)  +  N^{2/3}yq^{-2} )$, $w' = ( 2N(t +d_1) , 2Nt  +  N^{2/3}yq^{-2} )$, and $w'' = ( 2Nt , 2Nt  +  N^{2/3}yq^{-2} )$. Then,
    \begin{equation}
    \begin{split}
        &\PP(|\log Z_{1,w} - \log Z_{1,w''}-4pNd_1| \geq Kd_1^{1/3}N^{1/3}2^{1/2}q^{-1}\sigma_p )\\
        &\leq \PP(|\log Z_{1,w} - \log Z_{1,w''}-4pNd_1| \geq CK)\\
        &\leq \PP(|\log Z_{1,w} - \log Z_{1,w''}| \geq CK - 4pc_3)\\
        &\leq \PP(|\log Z_{1,w} - \log Z_{1,w'}| \geq \frac{1}{2}C'K ) +\PP(|\log Z_{1,w'} - \log Z_{1,w''}| \geq \frac{1}{2}C'K )\\
    \end{split}
    \end{equation}
where $C = c_3^{1/3}2^{1/2}q^{-1}\sigma_p$ and $C',K_0$ are chosen sufficiently large so that the constant $4pc_0$ is absorbed into the $C'K$ term for all $K \geq K_0$. We then apply \cite[Theorem 3.28]{basu2024temporal} with $s = K^{2/3}$. Together with the random walk approximation for sub-exponential variables, we have that
\begin{equation}
 \PP(|\log Z_{1,w} - \log Z_{1,w'}| \geq \frac{1}{2}C'K ) +\PP(|\log Z_{1,w'} - \log Z_{1,w''}| \geq \frac{1}{2}C'K ) \leq C_4e^{-C_5K}.
\end{equation}

For $K \geq a_0N^{1/2}$, we note that $|\Lambda(w) - \Lambda(w')|$ can be bounded by some large constant independent of $K$, we can directly bound the probability as follows:
 \begin{equation}
    \begin{split}
        &\PP(|\log Z_{1,w} - \log Z_{1,w''}-4pNd_1| \geq CK)\\
        &\leq \PP(|\log Z_{1,w} - \Lambda(w)| \geq C'N^{1/2} ) + \PP(|\log Z_{1,w''} -\Lambda(w')| \geq C'N^{1/2} )\\
        & \leq C_6e^{-C_5K^{1/10}}.
    \end{split}
    \end{equation}
The last inequality follows from \cite[Proposition 3.6]{basu2024temporal} and \cite[Proposition 3.8]{basu2024temporal}.

We now turn to the proof of equation (\ref{equation_spatial}).  Let $w = ( 2Nt ,  2Nt  + N^{2/3}yq^{-2} )$ and $w' =  ( 2Nt ,  2Nt  +  N^{2/3}(y + d_2)q^{-2} )$. The probability in equation (\ref{equation_spatial}) is equivalent to the following expression:
    \begin{equation}
        \PP(|\log Z_{1,w'} - \log Z_{1,w} - pN^{2/3}d_2q^{-2}| \geq K d_2^{1/3}N^{1/3}2^{1/2}q^{-1}\sigma_p).
    \end{equation}
    For $K \leq a_0^{3/2}N^{1/2}$ where $a_0$ is the constant from \cite[Theorem 3.28]{basu2024temporal}, we still use the random walk approximation. Let $ a=  d_2N^{2/3}q^{-2} $ and for all $k \in \lb 0, a \rb$, define
    \begin{equation}
 \log Z_{1,w+(k,0)} - \log Z_{1,w} = \sum_{i=1}^{k} \log Z_{1, w + (i,0)} - \log Z_{1, w + (i-1,0)} = S_k. 
    \end{equation}
    We can view each summand as a step in the walk and \cite[Theorem 3.28]{basu2024temporal} allows us to upper and lower bound this walk by i.i.d. random walks $S^1, S^2$ where $S^2_k + \log \frac{9}{10} \leq S_k \leq S^1_k + \log \frac{10}{9}$ for all $k\in \lb 0, a \rb$  with high probability. Let $\rho \in (0, \theta)$ such that $w = 2N \xi[\rho]$. We set the perturbed parameters to be 
    \begin{equation}
    \lambda = \rho + q_0K^{2/3}N^{-1/3}
    \end{equation}
    \begin{equation}
    \eta = \rho - q_0K^{2/3}N^{-1/3}.
    \end{equation}
    Thus, the distribution of the steps of $S^1$ is given by $\log(\text{Ga}^{-1}(\theta - \lambda))$ and the distribution of the steps of $S^2$ is given by $\log(\text{Ga}^{-1}(\theta - \eta))$. Let $A$ denote the event that $S^2_k + \log \frac{9}{10} \leq S_k \leq S^1_k + \log \frac{10}{9}$ for all $k\in \lb 0, a \rb$. By \cite[Theorem 3.28]{basu2024temporal}, we know that $\PP(A^c) \leq e^{-CK^2}$. Thus, it suffices to bound the i.i.d. random walks with appropriate probability tail bound. Let $m(w)$ be the slope of vector $w$. Then,
    \begin{equation}
    |m(w)-1| =\left|\frac{-yN^{2/3}q^{-2}}{yN^{2/3}q^{-2} + 2Nt} \right|\leq 2^{-1}t_0|y|q^2N^{-1/3}.
    \end{equation}
    Apply \cite[Proposition 3.2]{basu2024temporal}, we get that 
    \begin{equation}
    \left|\rho - \frac{\theta}{2}\right| \leq  C'N^{-1/3}.
    \end{equation}
    Recall that for $X \sim \log(\text{Ga}^{-1}(\theta - z))$, $\E[X] = -\Psi(\theta - z)$ and $p = -\Psi(\theta/2)$. Then,
    \begin{equation}
    \begin{split}
        \left|\E[S_a^1] - pN^{2/3}d_2q^{-2}\right| &= \left|-a\Psi(\theta-\lambda) +a\Psi\left(\frac{\theta}{2}\right)\right|\\
        &= a \left| \Psi\left(\frac{\theta}{2}-\rho + \frac{\theta}{2}-q_0K^{2/3}N^{-1/3}\right) - \Psi\left(\frac{\theta}{2}\right)\right|\\
        &\leq ac'(K^{2/3} + C')N^{-1/3}\\
        &\leq d_2C''K^{2/3}N^{1/3}
    \end{split}
    \end{equation}
    where we absorb the constant $C'$ into $C''K^{2/3}$ by requiring $K \geq K_0$ for some large constant $K_0$.
    Thus, 
    \begin{equation}
    \begin{split}
        &\PP(S_a^{1} - pN^{2/3}d_2q^{-2} \geq K d_2^{1/3}N^{1/3}2^{1/2}q^{-1}\sigma_p)\\
        &= \PP(S_a^1 - \E[S_a^1]+\E[S_a^1] - pN^{2/3}d_2q^{-2} \geq K d_2^{1/3}N^{1/3}2^{1/2}q^{-1}\sigma_p).
    \end{split}
    \end{equation}
    We choose $K_0$ sufficiently large and $r_0$ sufficiently small such that $d_2C''K^{2/3}N^{1/3}$ is absorbed into the term $K d_2^{1/3}N^{1/3}2^{-1/2}q^{-1}\sigma_p$. The remaining task is to bound the random walk with i.i.d. subexponential variable, which yields the desired tail estimate. Everything is analogous for $S^2.$ 

    Lastly, for $K \geq a_0^{3/2}N^{1/2}$, set $K = zN^{1/2}$ where $z \geq a_0^{3/2}$. The probability in equation (\ref{equation_spatial}) is bounded by the following:
    \begin{equation}\label{equation+}
    \PP(|\log Z_{1,w'} - \log Z_{1,w} - pN^{2/3}d_2q^{-2}|\geq Czd_2^{1/3}N^{5/6}).
    \end{equation}
    Absorb $pN^{2/3}d_2q^{-2}$ into $C d_2^{1/3}N^{5/6}$. Then, we can bound (\ref{equation+}) by
    \[
    \PP(|\log Z_{1,w'} - 
   \Lambda(w')|\geq C' zd_2^{1/3}N^{5/6}) + \PP(|\log Z_{1,w} - 
    \Lambda(w)|\geq C'z d_2^{1/3}N^{5/6})
    \]
    because $|\Lambda(w') - \Lambda(w)|$ can be bounded by $C''N^{2/3}$ and thus absorbed into $C' zd_2^{1/3}N^{5/6}$. Since $d_2N^{2/3}q^{-2} \in \Z_{\geq 0}$ and the inequality is trivial with $d_2N^{2/3}q^{-2} = 0$, we can assume that $d_2^{1/3}N^{2/9} \geq q^{2/3}$. Thus,
    \[
    \begin{split}
   & \PP(|\log Z_{1,w'} - 
    \Lambda(w')|\geq C' zd_2^{1/3}N^{5/6}) + \PP(|\log Z_{1,w} - 
   \Lambda(w)|\geq C' zd_2^{1/3}N^{5/6})
    \\
    &\leq \PP(|\log Z_{1,w'} - 
    \Lambda(w')|\geq C'''zN^{1/3}N^{5/18}) + \PP(|\log Z_{1,w} - 
    \Lambda(w)|\geq C'''zN^{1/3}N^{5/18}).
    \end{split}
    \]
    From this point, it is clear that we can apply the one-point tail bound in \cite[Proposition 3.6]{basu2024temporal} and \cite[Proposition 3.8]{basu2024temporal} again and establish the desired tail bound.
\end{proof}

\bibliographystyle{amsalpha}
\bibliography{main.bib}
\end{document}